\documentclass[reqno]{amsart}
\usepackage{amsmath,amsthm,amscd,amsfonts,amssymb,enumerate,mathtools}
\usepackage{graphicx}		

\makeatletter
\renewcommand\normalsize{
   \@setfontsize\normalsize\@xpt\@xiipt
   \abovedisplayskip 2\p@ \@plus2\p@ \@minus4\p@           
   \abovedisplayshortskip \z@ \@plus3\p@                   
   \belowdisplayshortskip 2\p@ \@plus2\p@ \@minus4\p@      
   \belowdisplayskip \abovedisplayskip
   \let\@listi\@listI}
\makeatother

\allowdisplaybreaks[4]
\hfuzz=\maxdimen
\vfuzz=\maxdimen
\theoremstyle{plain}
\newtheorem{theorem}{Theorem}[section]

\newtheorem{lemma}[theorem]{Lemma}
\newtheorem{corollary}[theorem]{Corollary}
\theoremstyle{definition}
\newtheorem{definition}[theorem]{Definition}
\theoremstyle{remark}
\newtheorem{remark}[theorem]{Remark}

\numberwithin{equation}{section}

\begin{document}
\title[Lie $n$-derivations of unital algebras with nontrivial idempotents]{Characterizations of Lie $n$-derivations of unital algebras with nontrivial idempotents}
\author[Y. Ding]{Yana Ding}
\address[Yana Ding]{Department of Mathematics, East China University of Science and Technology, Shanghai, P. R. China}
\email{dingyana@mail.ecust.edu.cn}
\author[J. Li]{Jiankui Li $^\ast$$^\dag$}
\address[Jiankui Li]{Department of Mathematics, East China University of Science and Technology, Shanghai, P. R. China}
\email{jiankuili@yahoo.com}
\thanks{$^\ast$Corresponding author. E-mail address: jiankuili@yahoo.com}
\thanks{$^\dag$This research was partially supported by National Natural Science Foundation of China (Grant No. 11371136).}
\begin{abstract}
Let $\mathcal{A}$ be a unital algebra with a nontrivial idempotent $e$, and $f=1-e$. Suppose that $\mathcal{A}$ satisfies that $exe\cdot e\mathcal{A}f=\{0\}=f\mathcal{A}e\cdot exe$ implies $exe=0$ and $e\mathcal{A}f\cdot fxf=\{0\}=fxf\cdot f\mathcal{A}e$ implies $fxf=0$ for each $x$ in $\mathcal{A}$. We obtain the (necessary and) sufficient conditions for a Lie $n$-derivation $\varphi$ on $\mathcal{A}$ to be of the form $\varphi=d+\delta+\gamma$, where $d$ is a derivation on $\mathcal{A}$, $\delta$ is a singular Jordan derivation on $\mathcal{A}$ and $\gamma$ is a linear mapping from $\mathcal{A}$ into the centre $\mathcal{Z}(\mathcal{A})$ vanishing on all $(n-1)-$th commutators of $\mathcal{A}$. In particular, we also discuss the (necessary and) sufficient conditions for a Lie $n$-derivation $\varphi$ on $\mathcal{A}$ to be standard, i.e., $\varphi=d+\gamma$.
\end{abstract}
\maketitle

\section{\bf Introduction}

\indent Let $\mathcal{A}$ be a unital algebra over a unital commutative ring $\mathcal{R}$. The algebra $\mathcal{A}$ is call to be \emph{$n$-torsion free} if $nx=0$ implies $x=0$ for some positive integer $n$ and each $x$ in $\mathcal{A}$, and is call to be \emph{torsion-free} if $nx=0$ implies $x=0$ for each positive integer $n$ and each $x$ in $\mathcal{A}$. A linear mapping $\delta$ on $\mathcal{A}$ is called a \emph{derivation} if $\delta(xy)=\delta(x)y+x\delta(y)$ for each $x,y$ in $\mathcal{A}$, is called a \emph{Jordan derivation} if $\delta(x\circ y)=\delta(x)\circ y+x\circ\delta(y)$ for each $x,y$ in $\mathcal{A}$, is called a \emph{Lie derivation} if $\delta([x,y])=[\delta(x),y]+[x,\delta(y)]$ for each $x,y$ in $\mathcal{A}$, and is called a \emph{Lie triple derivation} if $\delta([[x,y],z])=[[\delta(x),y],z]+[[x,\delta(y)],z]+[[x,y],\delta(z)]$ for each $x,y,z$ in $\mathcal{A}$, where $x\circ y=xy+yx$ and $[x,y]=xy-yx$ for each $x,y$ in $\mathcal{A}$. A derivation $\delta$ is called an \emph{inner derivation} if there exists some $a$ in $\mathcal{A}$ such that $\delta(x)=ax-xa$ for each $x$ in $\mathcal{A}$. Now we define a sequence of polynomials as follows:
\begin{align*}
  p_1(x_1)&=x_1, \\
  p_n(x_1,x_2,...,x_n)&=[p_{n-1}(x_1,x_2,...,x_{n-1}),x_n]
\end{align*}
\noindent for each $x_1,x_2,...,x_n\in\mathcal{A}$ and each positive integer $n\geq 2$. Thus, $p_2(x_1,x_2)=[x_1,x_2]$ and $p_3(x_1,x_2,x_3)=[[x_1,x_2],x_3]$. For $n\geq 2$, $p_n(x_1,x_2,...,x_n)=[...[[x_1,x_2],x_3],...,x_n]$ is also called an \emph{$(n-1)-$}th \emph{commutator} of $x_1,x_2,...,x_n\in\mathcal{A}$.
A linear mapping $\delta$ on $\mathcal{A}$ is called a \emph{Lie $n$-derivation} ($n\geq 2$) if
$$\delta(p_n(x_1,x_2,...,x_n))=\sum_{i=1}^{n}p_n(x_1,...,x_{i-1},\delta(x_i),x_{i+1},...,x_n)$$
\noindent for each $x_1,x_2,...,x_n\in\mathcal{A}$. Thus, $\delta$ is a Lie derivation when $n=2$, and is a Lie triple derivation when $n=3$. The notion of Lie $n$-derivations is firstly proposed by Abdullaev in \cite{Abdullaev}, where describe the form of Lie n-derivations of a certain von Neumann algebra (or of its skew-adjoint part). A Lie $n$-derivation $\delta$ on $\mathcal{A}$ is called to be \emph{standard} if $\delta=h+\tau$, where $h$ is a derivation on $\mathcal{A}$ and $\tau$ is a linear mapping from $\mathcal{A}$ into its centre $\mathcal{Z}(\mathcal{A})$ vanishing on all $(n-1)-$th commutators of $\mathcal{A}$.\\

\indent Let $e$ be a nontrivial idempotent in $\mathcal{A}$, and $f=1-e$. Then $\mathcal{A}$ can be represented in the so called Pierce decomposition form
\begin{equation}\label{1.5}
  \mathcal{A}=e\mathcal{A}e+e\mathcal{A}f+f\mathcal{A}e+f\mathcal{A}f
\end{equation}
where $e\mathcal{A}e$ is a subalgebra with unit $e$, $f\mathcal{A}f$ is a subalgebra with unit $f$, $e\mathcal{A}f$ is an $(e\mathcal{A}e,f\mathcal{A}f)$-bimodule, and $f\mathcal{A}e$ is an $(f\mathcal{A}f,e\mathcal{A}e)$-bimodule. In this paper, we study the conditions under which a Lie $n$-derivation on $\mathcal{A}$ is standard. Benkovi\v c and \v Sirovnik \cite{Benkovic-Sirovnik} consider Jordan derivations on unital algebras with nontrivial idempotents, and introduce the notion of singular Jordan derivations which comes out to be very important in study of mappings on such algebras. Benkovi\v c \cite{Benkovic} shows several sufficient (and necessary) conditions for a Lie triple derivation on $\mathcal{A}$ to be expressed as the sum of a derivation, a singular Jordan derivation and a linear mapping from $\mathcal{A}$ into the centre $\mathcal{Z}(\mathcal{A})$ vanishing on all second commutators of $\mathcal{A}$. Wang \cite{Wang} discusses the sufficient conditions for a Lie $n$-derivation on $\mathcal{A}$ to be expressed as the sum of a derivation, a singular Jordan derivation and a linear mapping from $\mathcal{A}$ into the centre $\mathcal{Z}(\mathcal{A})$ vanishing on all $(n-1)-$th commutators of $\mathcal{A}$. It is worth to mention that $\mathcal{A}$ is isomorphic to a generalized matrix algebra $\mathcal{G}=(A, M, N, B)$ (which is first introduced by Morita in \cite{Morita}), where $A$ and $B$ are two unital algebras, and $_AM_B$ and $_BN_A$ are two bimodules. Many papers discuss mappings on generalized matrix algebras such as \cite{Du-Wang-bider, Du-Wang-lieder, Li-Wei, Li-Wyk-Wei, Mokhtari-Vishki, Wang-Wang, Xiao-Wei}. With a quite common assumption that the bimodule $_AM_B$ is faithful which means that $aM=0$ implies $a=0$ for each $a\in A$ and that $Mb=0$ implies $b=0$ for each $b\in B$, the authors \cite{Mokhtari-Vishki, Wang-Wang} obtain sufficient conditions for Lie derivations and Lie $n$-derivations on generalized matrix algebras to be standard. In this paper, we consider a milder assumption which arises from \cite{Benkovic} that the Pierce decomposition (\ref{1.5}) satisfies
\begin{equation}\label{1.6}
\begin{array}{llll}
  exe\cdot e\mathcal{A}f=\{0\}=f\mathcal{A}e\cdot exe &\textrm{implies}&  exe=0 & \textrm{and}  \\
  e\mathcal{A}f\cdot fxf=\{0\}=fxf\cdot f\mathcal{A}e &\textrm{implies}&  fxf=0
\end{array}
\end{equation}
for each $x$ in $\mathcal{A}$. Important examples of unital algebras with nontrivial idempotents satisfying the property (\ref{1.6}) include triangular algebras, matrix algebras, algebras of all bounded linear operators of Banach space and prime algebras with nontrivial idempotents. \\

This paper is organized as follows. In Section \ref{S}, we consider $\mathcal{A}$ as a unital algebra with a nontrivial idempotent $e$ satisfying the property (\ref{1.6}). We discuss the (necessary and) sufficient conditions for a Lie $n$-derivation $\varphi$ on $\mathcal{A}$ to be of the form $\varphi=d+\delta+\gamma$, where $d$ is a derivation on $\mathcal{A}$, $\delta$ is a singular Jordan derivation on $\mathcal{A}$ and $\gamma$ is a linear mapping from $\mathcal{A}$ into the centre $\mathcal{Z}(\mathcal{A})$ vanishing on all $(n-1)-$th commutators of $\mathcal{A}$, which improve the corresponding main results in \cite{Benkovic, Wang}.
In particular, we also discuss the (necessary and) sufficient conditions for Lie $n$-derivations to be standard.

In Section \ref{A}, as applications of the results in Section \ref{S}, we characterize Lie $n$-derivations on matrix algebras, triangular algebras, unital prime algebras with nontrivial idempotents and von Neumann algebras.\\

\section{\bf Main Results}\label{S}
In this section, we assume that $\mathcal{A}$ is a unital algebra with a nontrivial idempotent $e$. 
By the Pierce decomposition (\ref{1.5}), $\mathcal{A}$ can be represented as $\mathcal{A}=e\mathcal{A}e+e\mathcal{A}f+f\mathcal{A}e+f\mathcal{A}f$, where $f=1-e$. In \cite{Benkovic-Sirovnik}, Benkovi\v c and \v Sirovnik introduce the term \emph{singular Jordan derivations}, which turns out to play an important role in the study of mappings on unital algebras with nontrivial idempotents.

\begin{definition}
A Jordan derivation $\delta$ on $\mathcal{A}$ is a \emph{singular Jordan derivation} if
\begin{equation}\label{S.1}
\begin{array}{lllll}
  \delta(e\mathcal{A}e)=\{0\}, & \delta(f\mathcal{A}f)=\{0\}, & \delta(e\mathcal{A}f)\subseteq f\mathcal{A}e & \textrm{and} & \delta(f\mathcal{A}e)\subseteq e\mathcal{A}f.
\end{array}
\end{equation}
\end{definition}
\noindent It's obvious that singular Jordan derivations on $\mathcal{A}$ is zero when $\mathcal{A}$ is a triangular algebra, since $f\mathcal{A}e=\{0\}$.

\begin{lemma}\cite[Proposition 2.1 and Remark 2.2]{Benkovic}\label{S_2}
If $\mathcal{A}$ satisfies the property (\ref{1.6}), then
\begin{description}
\item[(i)] $\mathcal{Z}(\mathcal{A})=\left\{\begin{array}{l|l}
                            a+b & \begin{array}{l}
                                    a\in e\mathcal{A}e,~b\in f\mathcal{A}f,\\
                                    am=mb,~ta=bt~for~each~m\in e\mathcal{A}f~and~t\in f\mathcal{A}e
                                    \end{array}
                            \end{array}\right\}.$
\item[(ii)] There exists a unique algebra isomorphism $\tau$ from $e\mathcal{Z}(\mathcal{A})e$ to $f\mathcal{Z}(\mathcal{A})f$, such that for each $a\in e\mathcal{Z}(\mathcal{A})e$ we have that $am=m\tau(a)$ and $ta=\tau(a)t$ for each $m\in e\mathcal{A}f$ and $t\in f\mathcal{A}e$. \\
(Thus, $a+\tau(a)\in\mathcal{Z}(\mathcal{A})$ for each $a\in e\mathcal{Z}(\mathcal{A})e$ and $\tau^{-1}(b)+b\in\mathcal{Z}(\mathcal{A})$ for each $b\in f\mathcal{Z}(\mathcal{A})f$.)
\item[(iii)] For $x\in\mathcal{A}$, if $[x,e\mathcal{A}f]=\{0\}$ and $[x,f\mathcal{A}e]=\{0\}$, then $exe+fxf\in\mathcal{Z}(\mathcal{A})$.
\end{description}
\end{lemma}

\begin{remark}\label{S_3}
Let $\varphi$ be a Lie $n$-derivation on $\mathcal{A}$. Similar to the proofs of \cite[Lemma 3.1]{Benkovic} and \cite[Theorem 2.1]{Wang}, we can assume that $\varphi$ satisfies $e\varphi(e)f=f\varphi(e)e=0$. Actually, let $x_0=e\varphi(e)f-f\varphi(e)e$ and $d$ be an inner derivation on $\mathcal{A}$ that $d(x)=[x,x_0]$ for each $x$ in $\mathcal{A}$. Clearly $\varphi'=\varphi-d$ is also a Lie $n$-derivation. Since
\begin{align*}
  \varphi'(e)=&\varphi(e)-[e,e\varphi(e)f-f\varphi(e)e]\\
             =&\varphi(e)-e\varphi(e)f-f\varphi(e)e\\
             =&e\varphi(e)e+f\varphi(e)f,
\end{align*}
\noindent we obtain $e\varphi'(e)f=f\varphi'(e)e=0$. Thus, it suffices to consider the Lie $n$-derivation $\varphi$ on $\mathcal{A}$ satisfying $e\varphi(e)f=f\varphi(e)e=0$.
\end{remark}

\begin{theorem}\label{S_4}
Let $\varphi$ be a Lie $n$-derivation on $\mathcal{A}$. Suppose that $\mathcal{A}$ is $2$- and $(n-1)$-torsion free, and that $\mathcal{A}$ satisfies the property (\ref{1.6}). Then $\varphi$ is of the form
\begin{equation}\label{S.2}
  \varphi=d+\delta+\gamma
\end{equation}
\noindent where $d$ is a derivation on $\mathcal{A}$, $\delta$ is a singular Jordan derivation on $\mathcal{A}$ and $\gamma$ is a linear mapping from $\mathcal{A}$ into the centre $\mathcal{Z}(\mathcal{A})$ vanishing on all $(n-1)-$th commutators of $\mathcal{A}$, if and only if
\begin{description}
\item[(i)] $f\varphi(e\mathcal{A}e)f\subseteq f\mathcal{Z}(\mathcal{A})f$ and $e\varphi(f\mathcal{A}f)e\subseteq e\mathcal{Z}(\mathcal{A})e$,
\item[(ii)] $e\varphi(tm)e+f\varphi(mt)f\in\mathcal{Z}(\mathcal{A})$ for each $m\in e\mathcal{A}f$ and $t\in f\mathcal{A}e$.
\end{description}
\end{theorem}

\begin{proof}
Suppose that $\varphi$ is of the form (\ref{S.2}) $\varphi=d+\delta+\gamma$. Let $\delta'=d+\delta$, then $\delta'$ is a Jordan derivation and
\begin{align}
  2\delta'(a)=&\delta'(e\circ a)=\delta'(e)a+a\delta'(e)+e\delta'(a)+\delta'(a)e  \label{S.3}, \\
  2\delta'(b)=&\delta'(f\circ b)=\delta'(f)b+b\delta'(f)+f\delta'(b)+\delta'(b)f  \label{S.4},
\end{align}
for each $a$ in $e\mathcal{A}e$ and $b$ in $f\mathcal{A}f$. Since $\mathcal{A}$ is $2$-torsion free, left and right multiplication of (\ref{S.3}) by $f$ implies that $f\delta'(a)f=0$ for each $a$ in $e\mathcal{A}e$, and left and right multiplication of (\ref{S.4}) by $e$ implies that $e\delta'(b)e=0$ for each $b$ in $f\mathcal{A}f$. Since $\gamma(\mathcal{A})\subseteq\mathcal{Z}(\mathcal{A})$, we have that
\begin{align*}
  f\varphi(a)f=&f\delta'(a)f+f\gamma(a)f=f\gamma(a)f\in f\mathcal{Z}(\mathcal{A})f,  \\
  e\varphi(b)e=&e\delta'(b)e+e\gamma(b)e=e\gamma(b)e\in e\mathcal{Z}(\mathcal{A})e,
\end{align*}
for each $a$ in $e\mathcal{A}e$ and $b$ in $f\mathcal{A}f$. Hence, (i) holds. For each $m$ in $e\mathcal{A}f$ and $t$ in $f\mathcal{A}e$, we have $$p_n(t,e,...,e,m)=p_{n-1}(t,e,...,e,m)=...=[t,m]=tm-mt.$$ Since $\gamma$ vanishes on all $(n-1)-$th commutators of $\mathcal{A}$, we have that $\gamma(tm-mt)=0$. We may as well assume that $\gamma(mt)=\gamma(tm)=a_0+b_0\in\mathcal{Z}(\mathcal{A})$ where $a_0\in e\mathcal{Z}(\mathcal{A})e$ and $b_0\in \mathcal{Z}(\mathcal{A})f$. Since $mt\in e\mathcal{A}e$, $tm\in f\mathcal{A}f$, we have that
\begin{align}
  \varphi(mt)=&d(mt)+\delta(mt)+\gamma(mt)=d(m)t+md(t)+a_0+b_0,  \label{S.5}\\
  \varphi(tm)=&d(tm)+\delta(tm)+\gamma(tm)=d(t)m+td(m)+a_0+b_0.  \label{S.6}
\end{align}
Left and right multiplication of (\ref{S.5}) by $f$ implies that $f\varphi(mt)f=b_0$, and left and right multiplication of (\ref{S.6}) by $e$ implies that $e\varphi(tm)e=a_0$. Thus, $e\varphi(tm)e+f\varphi(mt)f=a_0+b_0\in\mathcal{Z}(\mathcal{A})$, (ii) holds.\\

\indent Suppose that (i) and (ii) hold. According to Remark \ref{S_3}, it suffices to consider Lie $n$-derivation $\varphi$ on $\mathcal{A}$ satisfying $e\varphi(e)f=f\varphi(e)e=0$. Thus, $\varphi(e)=e\varphi(e)e+f\varphi(e)f$. We organize the following proof by a series of claims.\\

\noindent\ \textbf{Claim 1.}\indent For each $x\in\mathcal{A}$, we have $p_n(x,e,...,e)=(-1)^{n-1}exf+fxe$ and $p_n(x,f,...,f)=exf+(-1)^{n-1}fxe$.\\

It's obvious that $$p_n(x,e,...,e)=p_{n-1}([x,e],e,...,e)=p_{n-1}(-exf+fxe,e,...,e)=...=(-1)^{n-1}exf+fxe.$$ The case of $p_n(x,f,...,f)$ could be similarly proved.\\

\noindent$\begin{array}{lll}\textbf{Claim 2.}
            &\varphi(a)=e\varphi(a)e+f\varphi(a)f & \textrm{for each $a\in e\mathcal{A}e$,}\\
            &\varphi(b)=e\varphi(b)e+f\varphi(b)f & \textrm{for each $b\in f\mathcal{A}f$,}\\
            &\varphi(m)=e\varphi(m)f+f\varphi(m)e & \textrm{for each $m\in e\mathcal{A}f$,}\\
            &\varphi(t)=e\varphi(t)f+f\varphi(t)e & \textrm{for each $t\in f\mathcal{A}e$.}
            \end{array}$

For each $a$ in $e\mathcal{A}e$, since $[a,e]=0$ and $p_n(a,e,...,e)=0$, according to Claim 1, we have that
\begin{align*}
  0=&\varphi(p_n(a,e,...,e))\\
   =&p_n(\varphi(a),e,...,e)+p_n(a,\varphi(e),...,e)+\sum_{j=3}^{n}p_n(a,e,...,\varphi(e),...,e)\\
   =&(-1)^{n-1}e\varphi(a)f+f\varphi(a)e+(-1)^{n-2}e[a,\varphi(e)]f+f[a,\varphi(e)]e\\
   =&(-1)^{n-1}e\varphi(a)f+f\varphi(a)e.
\end{align*}
Left and right multiplying by $e$ and $f$ respectively in the above equations, we obtain that
\begin{equation}\label{S.7}
  e\varphi(a)f=f\varphi(a)e=0.
\end{equation}
Thus, $\varphi(a)\!=\!e\varphi(a)e+f\varphi(a)f$. For each $b$ in $f\mathcal{A}f$, since $[b,f]=0$ and $p_n(b,f,...,f)=0$, we can similarly prove that
\begin{equation}\label{S.8}
  e\varphi(b)f=f\varphi(b)e=0
\end{equation}
and $\varphi(b)=e\varphi(b)e+f\varphi(b)f$. For each $m$ in $e\mathcal{A}f$, according to Claim 1, we have that $p_n(m,e,...,e)=(-1)^{n-1}m$ and
\begin{align*}
  (-1)^{n-1}\varphi(m)=&\varphi(p_n(m,e,...,e))\\
   =&p_n(\varphi(m),e,...,e)+\sum_{j=2}^{n}p_n(m,e,...,\varphi(e),...,e)\\
   =&(-1)^{n-1}e\varphi(m)f+f\varphi(m)e+\sum_{j=2}^{n}(-1)^{j-2}p_{n-j+2}(m,\varphi(e),e,...,e)\\
   =&(-1)^{n-1}e\varphi(m)f+f\varphi(m)e+\sum_{j=2}^{n}(-1)^{j-2}(-1)^{n-j}[m,\varphi(e)]\\
   =&(-1)^{n-1}e\varphi(m)f+f\varphi(m)e+(-1)^{n-2}(n-1)[m,\varphi(e)].
\end{align*}
Left and right multiplying by $e$ and $f$ respectively in the above equations, under the assumption that $\mathcal{A}$ is $(n-1)$-torsion free, we obtain that
\begin{align}
  &e\varphi(m)e=f\varphi(m)f=0,\notag\\ 
  &f\varphi(m)e=(-1)^{n-1}f\varphi(m)e,\label{S.9-2}\\
  &[m,\varphi(e)]=0.\label{S.9-3}
\end{align}
Thus, $\varphi(m)=e\varphi(m)f+f\varphi(m)e$. For each $t$ in $f\mathcal{A}e$, by Claim 1, we have $p_n(t,e,...,e)=t$. We can similarly prove that $$\varphi(t)=(-1)^{n-1}e\varphi(t)f+f\varphi(t)e+(n-1)[t,\varphi(e)]$$ and
\begin{align}
  &e\varphi(t)e=f\varphi(t)f=0,\notag\\ 
  &e\varphi(t)f=(-1)^{n-1}e\varphi(t)f,\label{S.10-2}\\
  &[t,\varphi(e)]=0.\label{S.10-3}
\end{align}
Thus, $\varphi(t)=e\varphi(t)f+f\varphi(t)e$.

According to Lemma \ref{S_2}, there exists a unique algebra isomorphism $\tau$ from $e\mathcal{Z}(\mathcal{A})e$ to $f\mathcal{Z}(\mathcal{A})f$ such that for each $a\in e\mathcal{Z}(\mathcal{A})e$ we have that $am=m\tau(a)$ and $ta=\tau(a)t$ for each $m\in e\mathcal{A}f$ and $t\in f\mathcal{A}e$. For each $a\in e\mathcal{A}e$, $m\in e\mathcal{A}f$, $t\in f\mathcal{A}e$ and $b\in f\mathcal{A}f$, we define a linear mapping $d$ on $\mathcal{A}$ as follows:
\begin{equation}\label{S.11}
  \begin{array}{lllll}
  d(a)\!=\!e\varphi(a)e\!-\!\tau^{-1}(f\varphi(a)f),\!&\!d(b)\!=\!f\varphi(b)f\!-\!\tau(e\varphi(b)e),\!&\!d(m)\!=\!e\varphi(m)f\!&\! \!\textrm{and}\! \!&\!d(t)\!=\!f\varphi(t)e,
  \end{array}
\end{equation}
and a linear mapping $\delta$ on $\mathcal{A}$ as follows:
\begin{equation}\label{S.12}
  \begin{array}{lllll}
  \delta(a)\!=\!0,&\delta(b)\!=\!0,&\delta(m)\!=\!f\varphi(m)e& \textrm{and} &\delta(t)\!=\!e\varphi(t)f.
  \end{array}
\end{equation}
\noindent  Denote $\gamma=\varphi-d-\delta$. Then $\gamma$ is a linear mapping satisfying
\begin{equation}\label{S.13}
  \begin{array}{lllll}
  \gamma(a)\!=\!\tau^{-1}(f\varphi(a)f)\!+\!f\varphi(a)f,\!&\!\gamma(b)\!=\!e\varphi(b)e\!+\!\tau(e\varphi(b)e),\!&\!\gamma(m)\!=\!0\!&\! \!\textrm{and}\! \!&\!\gamma(t)\!=\!0
  \end{array}
\end{equation}
\noindent for each $a\in e\mathcal{A}e$, $m\in e\mathcal{A}f$, $t\in f\mathcal{A}e$ and $b\in f\mathcal{A}f$.
By (i) and Lemma \ref{S_2}, $\gamma$ maps $\mathcal{A}$ into $\mathcal{Z}(\mathcal{A})$.\\

\noindent$\begin{array}{lll}\textbf{Claim 3}
            &\varphi(am)=[\varphi(a),m]+a\circ\varphi(m) & for~each~a\in e\mathcal{A}e, m\in e\mathcal{A}f,\\
            &\varphi(mb)=[m,\varphi(b)]+\varphi(m)\circ b & for~each~b\in f\mathcal{A}f, m\in e\mathcal{A}f,\\
            &\varphi(ta)=[t,\varphi(a)]+\varphi(t)\circ a & for~each~a\in e\mathcal{A}e, t\in f\mathcal{A}e,\\
            &\varphi(bt)=[\varphi(b),t]+b\circ\varphi(t) & for~each~b\in f\mathcal{A}f, t\in f\mathcal{A}e,\\
            &\varphi(mt)-\varphi(tm)=[\varphi(m),t]+[m,\varphi(t)] & for~each~m\in e\mathcal{A}f, t\in f\mathcal{A}e,\\
            &(-1)^{n-2}[\varphi(m_1),m_2]+[m_1,\varphi(m_2)]=0 & for~each~m_1,m_2\in e\mathcal{A}f,\\
            &(-1)^{n-2}[\varphi(t_1),t_2]+[t_1,\varphi(t_2)]=0 & for~each~t_1,t_2\in f\mathcal{A}e.
            \end{array}$

For each $a$ in $e\mathcal{A}e$ and $m$ in $e\mathcal{A}f$, since $[a,m]=am\in e\mathcal{A}f$, according to Claim 1, we have that $p_n(a,m,e,...,e)=(-1)^{n-2}am$ and
\begin{align*}
  (-1)^{n-2}\varphi(am)=&\varphi(p_n(a,m,e,...,e))\\
   =&p_n(\varphi(a),m,e,...,e)+p_n(a,\varphi(m),e,...,e)+\sum_{j=3}^{n}p_n(a,m,e,...,\varphi(e),...,e)\\
   =&(-1)^{n-2}e[\varphi(a),m]f+(-1)^{n-2}e[a,\varphi(m)]f+f[a,\varphi(m)]e\\
   &+\sum_{j=3}^{n}(-1)^{j-3}p_{n-j+3}(a,m,\varphi(e),e,...,e)\\
   =&(-1)^{n-2}[\varphi(a),m]+(-1)^{n-2}a\varphi(m)-\varphi(m)a+(-1)^{n-3}(n-2)[am,\varphi(e)].
\end{align*}
Left and right multiplying by $e$ and $f$ respectively in the above equations, we obtain that
\begin{align*}
  f\varphi(am)e&=(-1)^{n-1}\varphi(m)a,\\
  e\varphi(am)f&=[\varphi(a),m]+a\varphi(m)-(n-2)[am,\varphi(e)].
\end{align*}
Associating with (\ref{S.9-2}) and (\ref{S.9-3}), we have $[am,\varphi(e)]=0$ and $f\varphi(am)e=(-1)^{n-1}f\varphi(m)e\cdot eae=f\varphi(m)e\cdot eae=\varphi(m)a$. Thus, $$\varphi(am)=e\varphi(am)f+f\varphi(am)e=[\varphi(a),m]+a\circ\varphi(m)$$
for each $a$ in $e\mathcal{A}e$ and $m$ in $e\mathcal{A}f$.
Let's make similar discussions on $mb$, $ta$ and $bt$. For each $a$ in $e\mathcal{A}e$, $m$ in $e\mathcal{A}f$, $t$ in $f\mathcal{A}e$ and $b$ in $f\mathcal{A}f$, since $[m,b]=mb\in e\mathcal{A}f$, $[t,a]=ta\in f\mathcal{A}e$ and $[b,t]=bt\in f\mathcal{A}e$, we have that
\begin{equation*}
  \begin{array}{llll}
  p_n(m,b,e,...,e)=(-1)^{n-2}mb, & p_n(t,a,e,...,e)=ta & \textrm{and} & p_n(b,t,e,...,e)=bt.
  \end{array}
\end{equation*}
It follows that
\begin{align*}
  (-1)^{n-2}\varphi(mb)&=(-1)^{n-2}\varphi(m)b-b\varphi(m)+(-1)^{n-2}[m,\varphi(b)]+(-1)^{n-3}(n-2)[mb,\varphi(e)],\\
   \varphi(ta)&=-(-1)^{n-2}a\varphi(t)+\varphi(t)a+[t,\varphi(a)]+(n-2)[ta,\varphi(e)],\\
   \varphi(bt)&=[\varphi(b),t]-(-1)^{n-2}\varphi(t)b+b\varphi(t)+(n-2)[bt,\varphi(e)].
\end{align*}
Left and right multiplying by $e$ and $f$ respectively in the above equations, and associating with (\ref{S.9-2}), (\ref{S.9-3}), (\ref{S.10-2}) and (\ref{S.10-3}), we obtain that
  \begin{align*}
    &[mb,\varphi(e)]=[ta,\varphi(e)]=[bt,\varphi(e)]=0, \\
    &f\varphi(mb)e=b\varphi(m),\ e\varphi(ta)f=a\varphi(t)\ and\ e\varphi(bt)f=\varphi(t)b,\\
    &\varphi(mb)=[m,\varphi(b)]+\varphi(m)\circ b,\ \varphi(ta)=[t,\varphi(a)]+\varphi(t)\circ a\ and\ \varphi(bt)=[\varphi(b),t]+b\circ\varphi(t)
  \end{align*}
for each $a$ in $e\mathcal{A}e$, $m$ in $e\mathcal{A}f$, $t$ in $f\mathcal{A}e$ and $b$ in $f\mathcal{A}f$.

When $n=2$, for each $m,m_1,m_2$ in $e\mathcal{A}f$ and $t,t_1,t_2$ in $f\mathcal{A}e$, it is obvious that
\begin{align*}
&\varphi(mt)-\varphi(tm)=\varphi([m,t])=[\varphi(m),t]+[m,\varphi(t)],\\
&(-1)^{n-2}[\varphi(m_1),m_2]+[m_1,\varphi(m_2)]=[\varphi(m_1),m_2]+[m_1,\varphi(m_2)]=\varphi([m_1,m_2])=0,\\
&(-1)^{n-2}[\varphi(t_1),t_2]+[t_1,\varphi(t_2)]=[\varphi(t_1),t_2]+[t_1,\varphi(t_2)]=\varphi([t_1,t_2])=0.
\end{align*}

If $n\geq 3$, for each $m$ in $e\mathcal{A}f$ and $t$ in $f\mathcal{A}e$, since
$p_n(m,e,...,e,t)=[p_{n-1}(m,e,...,e),t]=(-1)^{n-2}[m,t]$, then we have that
\begin{align*}
  (-1)^{n-2}(\varphi(mt)-\varphi(tm))=&\varphi(p_n(m,e,...,e,t))\\
   =&p_n(\varphi(m),e,...,e,t)+\sum_{j=2}^{n-1}p_n(m,e,...,\varphi(e),...,e,t)\\
   &+p_n(m,e,...,e,\varphi(t))\\
   =&[(-1)^{n-2}e\varphi(m)f+f\varphi(m)e,t]+(-1)^{n-3}(n-2)[[m,\varphi(e)],t]\\
   &+(-1)^{n-2}[m,\varphi(t)]\\
   =&(-1)^{n-2}[\varphi(m),t]+(-1)^{n-2}[m,\varphi(t)].
\end{align*}
For each $m_1,m_2$ in $e\mathcal{A}f$, since $p_n(m_1,e,...,e,m_2)=[p_{n-1}(m_1,e,...,e),m_2]=0$, we have that
\begin{align*}
  0=&\varphi(p_n(m_1,e,...,e,m_2))\\
   =&p_n(\varphi(m_1),e,...,e,m_2)+\sum_{j=2}^{n-1}p_n(m_1,e,...,\varphi(e),...,e,m_2)+p_n(m_1,e,...,e,\varphi(m_2))\\
   =&[(-1)^{n-2}e\varphi(m_1)f+f\varphi(m_1)e,m_2]+(-1)^{n-3}(n-2)[[m_1,\varphi(e)],m_2]+(-1)^{n-2}[m_1,\varphi(m_2)]\\
   =&[\varphi(m_1),m_2]+(-1)^{n-2}[m_1,\varphi(m_2)].
\end{align*}
For each $t_1,t_2$ in $f\mathcal{A}e$, since $p_n(t_1,e,...,e,t_2)=[p_{n-1}(t_1,e,...,e),t_2]=0$, we have that
\begin{align*}
  0=&\varphi(p_n(t_1,e,...,e,t_2))\\
   =&p_n(\varphi(t_1),e,...,e,t_2)+\sum_{j=2}^{n-1}p_n(t_1,e,...,\varphi(e),...,e,t_2)+p_n(t_1,e,...,e,\varphi(t_2))\\
   =&[(-1)^{n-2}e\varphi(t_1)f+f\varphi(t_1)e,t_2]+(n-2)[[t_1,\varphi(e)],t_2]+[t_1,\varphi(t_2)]\\
   =&(-1)^{n-2}[\varphi(t_1),t_2]+[t_1,\varphi(t_2)].
\end{align*}

Thus, for each $n\geq 2$, we conclude that
\begin{align*}
&\varphi(mt)-\varphi(tm)=[\varphi(m),t]+[m,\varphi(t)],\\
&(-1)^{n-2}[\varphi(m_1),m_2]+[m_1,\varphi(m_2)]=0,\\
&(-1)^{n-2}[\varphi(t_1),t_2]+[t_1,\varphi(t_2)]=0
\end{align*}
for each $m,m_1,m_2$ in $e\mathcal{A}f$ and $t,t_1,t_2$ in $f\mathcal{A}e$.\\

\noindent$\begin{array}{ll}\textbf{Claim 4}
            & d~\textrm{is a derivation}.
            \end{array}$

According to the definition (\ref{S.11}) of $d$, we have that
\begin{equation}\label{S.14}
\begin{array}{lllll}
 d(a)=ed(a)e, &d(m)=ed(m)f, &d(t)=fd(t)e& \textrm{and} &d(b)=fd(b)f
\end{array}
\end{equation}
for each $a$ in $e\mathcal{A}e$, $m$ in $e\mathcal{A}f$, $t$ in $f\mathcal{A}e$ and $b$ in $f\mathcal{A}f$.
For each $a$ in $e\mathcal{A}e$ and $m$ in $e\mathcal{A}f$, by Claim 3, we have $\varphi(am)=[\varphi(a),m]+a\circ\varphi(m)$. Thus, $$e\varphi(am)f=[\varphi(a),m]+a\varphi(m)=e\varphi(a)e\cdot emf-emf\cdot f\varphi(a)f+eae\cdot e\varphi(m)f.$$ By (i) and the definition of $\tau$, we know that $emf\cdot f\varphi(a)f=\tau^{-1}(f\varphi(a)f)\cdot emf$. Thus,
\begin{equation}\label{S.15}
     d(am)=e\varphi(am)f=(e\varphi(a)e-\tau^{-1}(f\varphi(a)f))\cdot emf+eae\cdot e\varphi(m)f=d(a)m+ad(m)
\end{equation}
for each $a$ in $e\mathcal{A}e$ and $m$ in $e\mathcal{A}f$. Make similar discussions on $mb$, $ta$ and $bt$, and we obtain that
\begin{equation}\label{S.16}
\begin{array}{llll}
d(mb)=md(b)+d(m)b, &d(ta)=td(a)+d(t)a & \textrm{and} &d(bt)=d(b)t+bd(t)
\end{array}
\end{equation}
for each $a$ in $e\mathcal{A}e$, $m$ in $e\mathcal{A}f$, $t$ in $f\mathcal{A}e$ and $b$ in $f\mathcal{A}f$.
For each $m$ in $e\mathcal{A}f$ and $t$ in $f\mathcal{A}e$, by Claim 3, we have $\varphi(mt)-\varphi(tm)=[\varphi(m),t]+[m,\varphi(t)]$. Thus,
\begin{align*}
  e\varphi(mt)e-e\varphi(tm)e=&\varphi(m)t+m\varphi(t)=e\varphi(m)f\cdot fte+emf\cdot f\varphi(t)e, \\
  -f\varphi(mt)f+f\varphi(tm)f=&t\varphi(m)+\varphi(t)m=fte\cdot e\varphi(m)f+f\varphi(t)e\cdot emf.
\end{align*}
Since $mt\in e\mathcal{A}e$ and $tm\in f\mathcal{A}f$, we obtain that
\begin{align*}
    d(m)t+md(t)=&e\varphi(m)f\cdot fte+emf\cdot f\varphi(t)e=e\varphi(mt)e-e\varphi(tm)e\\
        =&d(mt)+\tau^{-1}(f\varphi(mt)f)-e\varphi(tm)e,\\
    d(t)m+td(m)=&f\varphi(t)e\cdot emf+fte\cdot e\varphi(m)f=-f\varphi(mt)f-f\varphi(tm)f\\
        =&d(tm)+\tau(e\varphi(tm)e)-f\varphi(mt)f,
\end{align*}
By (ii) and Lemma \ref{S_2}, $\tau(e\varphi(tm)e)=f\varphi(mt)f$ and $e\varphi(tm)e=\tau^{-1}(f\varphi(mt)f)$. Thus,
\begin{equation}\label{S.17}
\begin{array}{lll}
d(m)t+md(t)=d(mt)& \textrm{and} &d(t)m+td(m)=d(tm)
\end{array}
\end{equation}
for each $m$ in $e\mathcal{A}f$ and $t$ in $f\mathcal{A}e$. By (\ref{S.14}), (\ref{S.15}), (\ref{S.16}), (\ref{S.17}) and \cite[Lemma 2.3]{Benkovic-Sirovnik}, we obtain that $d$ is a derivation.\\

\noindent$\begin{array}{ll}\textbf{Claim 5.}
            & \delta~\textrm{is a singular Jordan derivation}.
            \end{array}$

According to the definition (\ref{S.12}) of $\delta$, we only need to prove that $\delta$ is a Jordan derivation.
For each $a$ in $e\mathcal{A}e$, $m$ in $e\mathcal{A}f$, $t$ in $f\mathcal{A}e$ and $b$ in $f\mathcal{A}f$, by Claim 3, we know that
$$\begin{array}{lllll}
f\varphi(am)e=\varphi(m)a, &f\varphi(mb)e=b\varphi(m), &e\varphi(ta)f=a\varphi(t)& \textrm{and} &e\varphi(bt)f=\varphi(t)b.
\end{array}$$
In view of (\ref{S.12}), we obtain that
\begin{equation}\label{S.18}
\begin{array}{lllll}
\delta(am)=\delta(m)a, &\delta(mb)=b\delta(m), &\delta(ta)=a\delta(t)& \textrm{and} &\delta(bt)=\delta(t)b
\end{array}
\end{equation}
for each $a$ in $e\mathcal{A}e$, $m$ in $e\mathcal{A}f$, $t$ in $f\mathcal{A}e$ and $b$ in $f\mathcal{A}f$.

For each  $m$ in $e\mathcal{A}f$, if $n$ is even, then $2f\varphi(m)e=0$ by (\ref{S.9-2}). Since $\mathcal{A}$ is $2$-torsion free, $f\varphi(m)e=0$, i.e. $\delta(m)=0$. If $n$ is odd, then by Claim 3, we have $2[m,\varphi(m)]=0$. Since $\mathcal{A}$ is $2$-torsion free, we have that $[m,\varphi(m)]=0$. Left and right multiplication by $e$ and $f$ respectively implies that $m\varphi(m)=\varphi(m)m=0$, i.e. $m\delta(m)=\delta(m)m=0$. Thus,
\begin{equation}\label{S.19}
  m\delta(m)=\delta(m)m=0
\end{equation}
for each $n\geq 2$.

For each  $t$ in $f\mathcal{A}e$, if $n$ is even, then $2e\varphi(t)f=0$ by (\ref{S.10-2}). Since $\mathcal{A}$ is $2$-torsion free, $e\varphi(t)f=0$, i.e. $\delta(t)=0$. If $n$ is odd, then by Claim 3, we have $2[t,\varphi(t)]=0$. Since $\mathcal{A}$ is $2$-torsion free, we have that $[t,\varphi(t)]=0$. Left and right multiplication by $e$ and $f$ respectively implies that $t\varphi(t)=\varphi(t)t=0$, i.e. $t\delta(t)=\delta(t)t=0$. Thus,
\begin{equation}\label{S.20}
  t\delta(t)=\delta(t)t=0
\end{equation}
for each $n\geq 2$.

Let $x=a+m+t+b$ be an arbitrary element in $\mathcal{A}$ where $a,m,t,b$ are elements in $e\mathcal{A}e, e\mathcal{A}f,f\mathcal{A}e,f\mathcal{A}f$, respectively. By (\ref{S.12}), (\ref{S.18}), (\ref{S.19}) and (\ref{S.20}), we obtain that
\begin{align*}
  \delta(x^2)=&\delta((a+m+t+b)^2)\\
  =&\delta(m)a+b\delta(m)+a\delta(t)+\delta(t)b,\\
  x\delta(x)+\delta(x)x=&(a+m+t+b)\delta(a+m+t+b)+\delta(a+m+t+b)(a+m+t+b)\\
  =&b\delta(m)+a\delta(t)+\delta(m)a+\delta(t)b.
\end{align*}
So $\delta(x^2)=x\delta(x)+\delta(x)x$ for each $x$ in $\mathcal{A}$. Thus, $\delta$ is a singular Jordan derivation.\\

\noindent$\begin{array}{ll}\textbf{Claim 6.}
            & \gamma~\textrm{vanishes on all}~(n-1)-\textrm{th commutators of}~\mathcal{A}.
            \end{array}$

For each $x_1,x_2,...,x_n$ in $\mathcal{A}$, we have that
\begin{align*}
 \gamma(p_n(x_1,x_2,...,x_n))=&\varphi(p_n(x_1,x_2,...,x_n))-d(p_n(x_1,x_2,...,x_n))-\delta(p_n(x_1,x_2,...,x_n))\\
 =&\sum_{i=1}^{n}p_n(x_1,...,\varphi(x_i),...,x_n)-d(p_n(x_1,x_2,...,x_n))
 -\delta(p_n(x_1,x_2,...,x_n))\\
 =&\sum_{i=1}^{n}p_n(x_1,...,d(x_i)+\delta(x_i)+\gamma(x_i),...,x_n)-d(p_n(x_1,x_2,...,x_n))\\
 &-\delta(p_n(x_1,x_2,...,x_n)).
\end{align*}
Since $d$ is a derivation and $\gamma(\mathcal{A})\subseteq\mathcal{Z}(\mathcal{A})$, it follows that
\begin{equation}\label{S.21}
 \gamma(p_n(x_1,x_2,...,x_n))=\sum_{i=1}^{n}p_n(x_1,...,x_{i-1},\delta(x_i),x_{i+1},...,x_n)-\delta(p_n(x_1,x_2,...,x_n)).
\end{equation}
\indent If $n$ is even, then in view of (\ref{S.9-2}), (\ref{S.10-2}), (\ref{S.12}) and that $\mathcal{A}$ is $2$-torsion free, we obtain that
$\delta(m)=f\varphi(m)e=0$ and $\delta(t)=e\varphi(t)f=0$
for each $m$ in $e\mathcal{A}f$ and $t$ in $f\mathcal{A}e$. Thus, $\delta(\mathcal{A})=\{0\}$. By (\ref{S.21}), we have $\gamma(p_n(x_1,x_2,...,x_n))=0$ for each $x_1,x_2,...,x_n$ in $\mathcal{A}$.\\
\indent If $n$ is odd, then by Claim 3 and (\ref{S.12}), we have that
\begin{align*}
    &\delta(am)=f\varphi(am)e=f\varphi(m)e\cdot eae=\delta(m)a, \\
    &\delta(mb)=f\varphi(mb)e=fbf\cdot f\varphi(m)e=b\delta(m), \\
    &\delta(ta)=e\varphi(ta)f=eae\cdot e\varphi(t)f=a\delta(t), \\
    &\delta(bt)=e\varphi(bt)f=e\varphi(t)f\cdot fbf=\delta(t)b, \\
    &\delta(m_1)m_2+\delta(m_2)m_1=f\varphi(m_1)e\cdot em_2f+f\varphi(m_2)e\cdot em_1f=0,\\
    &m_2\delta(m_1)+m_1\delta(m_2)=em_2f\cdot f\varphi(m_1)e+em_1f\cdot f\varphi(m_2)e=0,\\
    &\delta(t_1)t_2+\delta(t_2)t_1=e\varphi(t_1)f\cdot ft_2e+e\varphi(t_2)f\cdot ft_1e=0,\\
    &t_2\delta(t_1)+t_1\delta(t_2)=ft_2e\cdot e\varphi(t_1)f+ft_1e\cdot e\varphi(t_2)f=0
\end{align*}
\noindent for each $a$ in $e\mathcal{A}e$, $m,m_1,m_2$ in $e\mathcal{A}f$, $t,t_1,t_2$ in $f\mathcal{A}e$ and $b$ in $f\mathcal{A}f$. It follows that
\begin{align*}
    &\delta(a_1a_2m)=\delta(m)a_1a_2=\delta(a_1m)a_2=\delta(a_2a_1m)=\delta(m)a_2a_1,\\
    &\delta(mb_1b_2)=b_1b_2\delta(m)=b_1\delta(mb_2)=\delta(mb_2b_1)=b_2b_1\delta(m),\\
    &\delta(ta_1a_2)=a_1a_2\delta(t)=a_1\delta(ta_2)=\delta(ta_2a_1)=a_2a_1\delta(t),\\
    &\delta(b_1b_2t)=\delta(t)b_1b_2=\delta(b_1t)b_2=\delta(b_2b_1t)=\delta(t)b_2b_1,\\
    &\delta(amb)=\delta(mb)a=b\delta(m)a,\\
    &\delta(bta)=\delta(ta)b=a\delta(t)b,\\
    &\delta(t_1)t_2m+mt_2\delta(t_1)=\delta(t_2mt_1+t_1mt_2)=mt_1\delta(t_2)+mt_2\delta(t_1)=0,\\
    &\delta(m_1)m_2t+tm_2\delta(m_1)=\delta(m_2tm_1+m_1tm_2)=tm_1\delta(m_2)+tm_2\delta(m_1)=0
\end{align*}
\noindent for each $a,a_1,a_2$ in $e\mathcal{A}e$, $m,m_1,m_2$ in $e\mathcal{A}f$, $t,t_1,t_2$ in $f\mathcal{A}e$ and $b,b_1,b_2$ in $f\mathcal{A}f$. Thus, for each $x_1=a_1+m_1+t_1+b_1$, $x_2=a_2+m_2+t_2+b_2$ and $x_3=a_3+m_3+t_3+b_3$ in $\mathcal{A}$ where
$a_1,a_2,a_3\in e\mathcal{A}e$, $m_1,m_2,m_3\in e\mathcal{A}f$, $t_1,t_2,t_3\in f\mathcal{A}e$ and $b_1,b_2,b_3\in f\mathcal{A}f$, we obtain that
\begin{align*}
  \delta([[x_1,x_2],x_3])
  =&\delta([[a_1+m_1+t_1+b_1,a_2+m_2+t_2+b_2],a_3+m_3+t_3+b_3]) \\
  =&\delta(t_1a_2a_3+b_1t_2a_3-t_2a_1a_3-b_2t_1a_3-b_3t_1a_2-b_3b_1t_2+b_3t_2a_1+b_3b_2t_1+a_1m_2b_3\\
  &+m_1b_2b_3-a_2m_1b_3-m_2b_1b_3-a_3a_1m_2-a_3m_1b_2+a_3a_2m_1+a_3m_2b_1),\\
  [[\delta(x_1),x_2],x_3]+&[[x_1,\delta(x_2)],x_3]+[[x_1,x_2],\delta(x_3)]\\
  =&[[\delta(a_1\!\!+\!m_1\!\!+\!t_1\!\!+\!b_1),a_2\!\!+\!m_2\!\!+\!t_2\!\!+\!b_2],a_3\!\!+\!m_3\!\!+\!t_3\!\!+\!b_3]\\
  &\!+\![[a_1\!+\!m_1\!+\!t_1\!+\!b_1,\delta(a_2\!+\!m_2\!+\!t_2\!+\!b_2)],a_3\!+\!m_3\!+\!t_3\!+\!b_3]\\
  &\!+\![[a_1\!+\!m_1\!+\!t_1\!+\!b_1,a_2\!+\!m_2\!+\!t_2\!+\!b_2],\delta(a_3\!+\!m_3\!+\!t_3\!+\!b_3)]\\
  =&\delta(t_1)b_2b_3-a_2\delta(t_1)b_3-\delta(t_2)b_1b_3+a_1\delta(t_2)b_3
  -a_3\delta(t_1)b_2+a_3a_2\delta(t_1)+a_3\delta(t_2)b_1\\
  &-a_3a_1\delta(t_2)+\delta(m_1)a_2a_3-b_2\delta(m_1)a_3-\delta(m_2)a_1a_3+b_1\delta(m_2)a_3
  -b_3\delta(m_1)a_2\\
  &+b_3b_2\delta(m_1)+b_3\delta(m_2)a_1-b_3b_1\delta(m_2).
\end{align*}
It follows that $\delta([[x_1,x_2],x_3])=[[\delta(x_1),x_2],x_3]+[[x_1,\delta(x_2)],x_3]+[[x_1,x_2],\delta(x_3)]$, i.e., $\delta$ is a Lie triple derivation. Since $n$ is odd, we can deduce that
\begin{align*}
    \delta(p_n(x_1,x_2,...,x_n)=&\delta([[p_n(x_1,x_2,...,x_{n-2}),x_{n-1}],x_n])\\
    =&[[\delta(p_{n-2}(x_1,x_2,...,x_{n-2})),x_{n-1}],x_n]+[[p_{n-2}(x_1,x_2,...,x_{n-2}),\delta(x_{n-1})],x_n]\\
    &+[[p_{n-2}(x_1,x_2,...,x_{n-2}),x_{n-1}],\delta(x_n)]\\
    =&p_{n-2}(\delta([[x_1,x_2],x_3]),x_4...,x_n)+\sum_{i=4}^{n}p_{n-2}([[x_1,x_2],x_3]),x_4,...,\delta(x_i),...,x_n)\\
    =&\sum_{i=1}^{n}p_n(x_1,...,x_{i-1},\delta(x_i),x_{i+1},...,x_n),
\end{align*}
i.e., $\delta$ is a Lie $n$-derivation. By (\ref{S.21}), we have $\gamma(p_n(x_1,x_2,...,x_n))=0$ for each $x_1,x_2,...,x_n$ in $\mathcal{A}$. Thus, Claim 6 holds.

\indent With the definitions (\ref{S.11}), (\ref{S.12}) and (\ref{S.13}) of $d,\delta,\gamma$ and Claims 4, 5 and 6, the proof is finished.
\end{proof}

\begin{remark}
The above theorem generalizes \cite[Theorem 3.4]{Benkovic} which considers Lie triple derivations on a unital algebra with a nontrivial idempotent $e$ satisfying the property (\ref{1.6}).
\end{remark}

\begin{corollary}\label{S_5}
Let $\varphi$ be a Lie $n$-derivation on $\mathcal{A}$. Suppose that $\mathcal{A}$ is $2$- and $(n-1)$-torsion free, and that $\mathcal{A}$ satisfies the property (\ref{1.6}). Then $\varphi$ is standard if and only if
\begin{description}
\item[(i)] $f\varphi(e\mathcal{A}e)f\subseteq f\mathcal{Z}(\mathcal{A})f$ and $e\varphi(f\mathcal{A}f)e\subseteq e\mathcal{Z}(\mathcal{A})e$,
\item[(ii)] $e\varphi(tm)e+f\varphi(mt)f\in\mathcal{Z}(\mathcal{A})$ for each $m\in e\mathcal{A}f$ and $t\in f\mathcal{A}e$,
\item[(iii)] $f\varphi(e\mathcal{A}f)e=e\varphi(f\mathcal{A}e)f=\{0\}$.
\end{description}
\end{corollary}

\begin{proof}
Suppose that $\varphi$ is standard. That is, there exists a derivation $d$ on $\mathcal{A}$ and a linear mapping $\gamma$ from $\mathcal{A}$ into $\mathcal{Z}(\mathcal{A})$ vanishing on all $(n-1)-$th commutators of $\mathcal{A}$, such that $\varphi=d+\gamma$. Let $\delta$ be a linear mapping on $\mathcal{A}$ and $\delta=0$. Then $\delta$ is a singular Jordan derivation and $\varphi=d+\delta+\gamma$. According to Theorem \ref{S_4}, we only need to prove (iii). For each $m$ in $e\mathcal{A}f$ and $t$ in $f\mathcal{A}e$, we have
$p_n(m,f,...,f)=p_{n-1}(m,f,...,f)=...=[m,f]=m$ and $p_n(t,e,...,e)=p_{n-1}(t,e,...,e)=...=[t,e]=t$. Since $\gamma$ vanishes on all $(n-1)-$th commutators of $\mathcal{A}$, we have that $\gamma(m)=\gamma(t)=0$. Thus,
\begin{align*}
  \varphi(m)=&d(em)+\gamma(m)=d(e)m+ed(m),\\
  \varphi(t)=&d(te)+\gamma(t)=d(t)e+td(e).
\end{align*}
Left and right multiplication by $e$ and $f$ respectively implies that
$f\varphi(m)e=0$ and $e\varphi(t)f=0$.
Hence, (iii) holds.\\
\indent Suppose that (i), (ii) and (iii) hold. According to Theorem \ref{S_4}, $\varphi$ is of the form (\ref{S.2}) $\varphi=d+\delta+\gamma$. By (iii) and the definition (\ref{S.12}) of $\delta$ in Theorem \ref{S_4}, we obtain that $\delta=0$.
\end{proof}

\begin{corollary}\label{S_6}
Let $\varphi$ be a Lie $n$-derivation on $\mathcal{A}$. Suppose that $n$ is even, and that $\mathcal{A}$ is a $2$- and $(n-1)$-torsion free algebra satisfying the property (\ref{1.6}). Then $\varphi$ is standard if and only if
\begin{description}
\item[(i)] $f\varphi(e\mathcal{A}e)f\subseteq f\mathcal{Z}(\mathcal{A})f$ and $e\varphi(f\mathcal{A}f)e\subseteq e\mathcal{Z}(\mathcal{A})e$,
\item[(ii)] $e\varphi(tm)e+f\varphi(mt)f\in\mathcal{Z}(\mathcal{A})$ for each $m\in e\mathcal{A}f$ and $t\in f\mathcal{A}e$.
\end{description}
\end{corollary}

\begin{proof}
According to Corollary \ref{S_5}, we only need to prove that $f\varphi(e\mathcal{A}f)e=e\varphi(f\mathcal{A}e)f=\{0\}$ if (i) and (ii) hold.\\
\indent Suppose that (i) and (ii) hold. For each $m$ in $e\mathcal{A}f$ and $t$ in $f\mathcal{A}e$, by (\ref{S.9-2}) and (\ref{S.10-2}), we have that
$$\begin{array}{lll}
    f\varphi(m)e=(-1)^{n-1}f\varphi(m)e& \textrm{and} &e\varphi(t)f=(-1)^{n-1}e\varphi(t)f.
\end{array}$$
Since $n$ is even and $\mathcal{A}$ is $2$-torsion free, we have $f\varphi(m)e=e\varphi(t)f=0$.
\end{proof}

\begin{lemma}\cite[Remark 3.2]{Benkovic-Sirovnik}\label{S_7}
Let $\delta$ be a singular Jordan derivation on $\mathcal{A}$.
\begin{description}
\item[(i)] $\delta$ is an antiderivation if and only if $\delta$ satisfies
\begin{equation}\label{S.22-1}
\delta(e\mathcal{A}f)\cdot e\mathcal{A}f=e\mathcal{A}f\cdot\delta(e\mathcal{A}f)=\delta(f\mathcal{A}e)\cdot f\mathcal{A}e=f\mathcal{A}e\cdot\delta(f\mathcal{A}e)=\{0\}.
\end{equation}
\item[(ii)] If $\mathcal{A}$ satisfies
\begin{equation}\label{S.22-2}
  e\mathcal{A}f\cdot f\mathcal{A}e=f\mathcal{A}e\cdot e\mathcal{A}f=\{0\},
\end{equation}
\noindent then $\delta$ is an antiderivation.
\end{description}
\end{lemma}

\begin{corollary}\label{S_8}
Let $\varphi$ be a Lie $n$-derivation on $\mathcal{A}$. Suppose that $\mathcal{A}$ is $2$- and $(n-1)$-torsion free, and that $\mathcal{A}$ satisfies the property (\ref{1.6}) and (\ref{S.22-2}). Then $\varphi$ is of the form
\begin{equation}\label{S.23}
  \varphi=d+\delta+\gamma
\end{equation}
\noindent where $d$ is a derivation on $\mathcal{A}$, $\delta$ is a singular Jordan derivation and antiderivation on $\mathcal{A}$, and $\gamma$ is a linear mapping from $\mathcal{A}$ into $\mathcal{Z}(\mathcal{A})$ vanishing on all $(n-1)-$th commutators of $\mathcal{A}$, if and only if $f\varphi(e\mathcal{A}e)f\subseteq f\mathcal{Z}(\mathcal{A})f$ and $e\varphi(f\mathcal{A}f)e\subseteq e\mathcal{Z}(\mathcal{A})e$.
\end{corollary}

\begin{proof}
Since $\mathcal{A}$ satisfies (\ref{S.22-2}), we have that $mt=tm=0$ and $e\varphi(tm)e+f\varphi(mt)f=0$ for each $m$ in $e\mathcal{A}f$ and $t$ in $f\mathcal{A}e$. Thus, the condition (ii) in Theorem \ref{S_4} holds. It follows that $\varphi$ is of the form (\ref{S.23}) $\varphi=d+\delta+\gamma$ where $\delta$ is a singular Jordan derivation if and only if $f\varphi(e\mathcal{A}e)f\subseteq f\mathcal{Z}(\mathcal{A})f$ and $e\varphi(f\mathcal{A}f)e\subseteq e\mathcal{Z}(\mathcal{A})e$. By Lemma \ref{S_7}, $\delta$ is also an antiderivation.
\end{proof}

\begin{remark}
Corollaries \ref{S_6} and \ref{S_8} generalize \cite[Corollaries 3.5, 3.6 and 3.7]{Benkovic} which considers Lie derivations and Lie triple derivations on a unital algebra with a nontrivial idempotent $e$ satisfying the property (\ref{1.6}).
\end{remark}

In \cite[Theorem 1]{Wang-Wang} and \cite[Theorem 2.1]{Wang}, the authors introduce the idea that $e\mathcal{A}e$ or $f\mathcal{A}f$ has no nonzero central ideal. We would mention that if $\mathcal{A}$ is a unital algebra with a nontrivial idempotent $e$ satisfying the property (\ref{1.6}), then $\mathcal{A}$ has no nonzero central ideal, but we cannot confirm that $e\mathcal{A}e$ or $f\mathcal{A}f$ has no nonzero central ideal.

\begin{lemma}\label{S_9}
Let $\varphi$ be a Lie $n$-derivation on $\mathcal{A}$. Suppose that $\mathcal{A}$ is $(n-1)$-torsion free, and that $\mathcal{A}$ satisfies the property (\ref{1.6}). If
\begin{description}
\item[(i)] $f\varphi(e\mathcal{A}e)f\subseteq f\mathcal{Z}(\mathcal{A})f$ and $e\varphi(f\mathcal{A}f)e\subseteq e\mathcal{Z}(\mathcal{A})e$,
\item[(ii)] $e\mathcal{A}e$ or $f\mathcal{A}f$ has no nonzero central ideal.
\end{description}
Then $\varphi$ satisfies $e\varphi(tm)e+f\varphi(mt)f\in\mathcal{Z}(\mathcal{A})$ for each $m\in e\mathcal{A}f$ and $t\in f\mathcal{A}e$.
\end{lemma}

\begin{proof}
Without loss of generality, we suppose that $e\mathcal{A}e$ has no nonzero central ideal. Since $\varphi$ is a Lie $n$-derivation on $\mathcal{A}$. Discussing similarly as Theorem \ref{S_4}, we obtain same results as Claims 1, 2 and 3 in Theorem \ref{S_4}. According to (i) and the definition (\ref{S.11}) of $d$, we conclude that for each $a$ in $e\mathcal{A}e$, $m$ in $e\mathcal{A}f$ and $t$ in $f\mathcal{A}e$,
\begin{align}
    &d(am)=d(a)m+ad(m),\notag\\
    &d(ta)=td(a)+d(t)a,\notag\\
    &d(m)t+md(t)=d(mt)+\tau^{-1}(f\varphi(mt)f)-e\varphi(tm)e.\label{S.23-24}
\end{align}
For each $a_1,a_2$ in $e\mathcal{A}e$, $m$ in $e\mathcal{A}f$ and $t$ in $f\mathcal{A}e$, it follows that
\begin{align*}
  &d(a_1a_2m)=d(a_1a_2)m+a_1a_2d(m),\\
  &d(a_1a_2m)=d(a_1)a_2m+a_1d(a_2m)=d(a_1)a_2m+a_1d(a_2)m+a_1a_2d(m),\\
  &d(ta_1a_2)=td(a_1a_2)+d(t)a_1a_2,\\
  &d(ta_1a_2)=ta_1d(a_2)+d(ta_1)a_2=ta_1d(a_2)+td(a_1)a_2+d(t)a_1a_2.
\end{align*}
Thus,
\begin{align*}
    (d(a_1a_2)-d(a_1)a_2-a_1d(a_2))m=0,\\
    t(d(a_1a_2)-a_1d(a_2)-d(a_1)a_2)=0.
\end{align*}
Since $\mathcal{A}$ satisfies the property (\ref{1.6}), we obtain that $$d(a_1a_2)=d(a_1)a_2+a_1d(a_2).$$
Denote that $\epsilon(m,t)=e\varphi(tm)e-\tau^{-1}(f\varphi(mt)f)$ for each $m$ in $e\mathcal{A}f$ and $t$ in $f\mathcal{A}e$. Then $\epsilon(m,t)=d(mt)-d(m)t-md(t)$ by (\ref{S.23-24}). According to (i), we have that $\epsilon(m,t)\in e\mathcal{Z}(\mathcal{A})e\subseteq\mathcal{Z}(e\mathcal{A}e)$.
Since
\begin{align*}
  \epsilon(am,t)=&d(amt)-d(am)t-amd(t)\\
  =&d(a)mt+ad(mt)-d(a)mt-ad(m)t-amd(t)\\
  =&a(d(mt)-d(m)t-md(t))\\
  =&a\epsilon(m,t)
\end{align*}
for each $a$ in $e\mathcal{A}e$, $m$ in $e\mathcal{A}f$ and $t$ in $f\mathcal{A}e$, we obtain that $\epsilon(m,t)$ is a central ideal of $e\mathcal{A}e$. According to the assumption that $e\mathcal{A}e$ has no nonzero central ideal, we have $\epsilon(m,t)=0$, i.e., $e\varphi(tm)e-\tau^{-1}(f\varphi(mt)f)=0$. Thus, $e\varphi(tm)e+f\varphi(mt)f\in\mathcal{Z}(\mathcal{A})$.\\
\indent The proof in case that $f\mathcal{A}f$ has no nonzero central ideal goes in a similar way.
\end{proof}

\begin{corollary}\label{S_10}
Let $\varphi$ be a Lie $n$-derivation on $\mathcal{A}$. Suppose that $\mathcal{A}$ is $2$- and $(n-1)$-torsion free, and that $\mathcal{A}$ satisfies the property (\ref{1.6}). If
\begin{description}
\item[(i)] $f\varphi(e\mathcal{A}e)f\subseteq f\mathcal{Z}(\mathcal{A})f$ and $e\varphi(f\mathcal{A}f)e\subseteq e\mathcal{Z}(\mathcal{A})e$,
\item[(ii)] $e\mathcal{A}e$ or $f\mathcal{A}f$ has no nonzero central ideal,
\end{description}
then $\varphi$ is of the form
\begin{equation}\label{S.24}
\varphi=d+\delta+\gamma
\end{equation}
\noindent where $d$ is a derivation on $\mathcal{A}$, $\delta$ is a singular Jordan derivation and antiderivation on $\mathcal{A}$, and $\gamma$ is a linear mapping from $\mathcal{A}$ into $\mathcal{Z}(\mathcal{A})$ vanishing on all $(n-1)-$th commutators of $\mathcal{A}$.
\end{corollary}

\begin{proof}
According to Theorem \ref{S_4} and Lemma \ref{S_9}, we only need to prove that $\delta$ is an antiderivation.\\
\indent Without loss of generality, we suppose that $e\mathcal{A}e$ has no nonzero central ideal. The following discussion is partially similar as Claim 6 in Theorem \ref{S_4}, and we will omit several complicated procedures. If $n$ is even, then in view of (\ref{S.9-2}), (\ref{S.10-2}), (\ref{S.12}) and that $\mathcal{A}$ is $2$-torsion free, we obtain that $\delta(m)=0$ and $\delta(t)=0$ for each $m$ in $e\mathcal{A}f$ and $t$ in $f\mathcal{A}e$. Thus, $\delta=0$.
If $n$ is odd, then in view of (\ref{S.12}) and Claim 3 in Theorem \ref{S_4}, we have that for each $m_1,m_2$ in $e\mathcal{A}f$ and $t_1,t_2$ in $f\mathcal{A}e$,
\begin{equation}\label{S.25}
\begin{array}{lll}
[\delta(m_1),m_2]+[\delta(m_2),m_1]=0 & \textrm{and} & [\delta(t_1),t_2]+[\delta(t_2),t_1]=0.
\end{array}
\end{equation}
Take arbitrary elements $m,m_1,m_2$ in $e\mathcal{A}f$. If $n\geq 3$, since $p_n(m_1,m_2,m,f,...,f)=0$, we have that
\begin{align*}
    0=&\varphi(p_n(m_1,m_2,m,f,...,f))\\
    =&[[\varphi(m_1),m_2],m]+[[m_1,\varphi(m_2)],m]\\
    =&[[\varphi(m_1),m_2]+[m_1,\varphi(m_2)],m]\\
    =&[[\delta(m_1),m_2]+[m_1,\delta(m_2)],m].
\end{align*}
Or if $n=2$, then
\begin{align*}
    0=&\varphi([[m_1,m_2],m])\\
    =&[\varphi([m_1,m_2]),m]\\
    =&[[\varphi(m_1),m_2]+[m_1,\varphi(m_2)],m]\\
    =&[[\delta(m_1),m_2]+[m_1,\delta(m_2)],m].
\end{align*}
Thus, for each $n\geq 2$ and for each $m,m_1,m_2$ in $e\mathcal{A}f$, we have that
\begin{equation}\label{S.26}
[[\delta(m_1),m_2]+[m_1,\delta(m_2)],m]=0.
\end{equation}
Similarly, since $p_n(m_1,m_2,t,e,...,e)=p_n(t_1,t_2,m,f,...,f)=p_n(t_1,t_2,t,e,...,e)=0$ when $n\geq 3$, we can conclude that
\begin{align}
[[\delta(m_1),m_2]+[m_1,\delta(m_2)],t]=&0,\label{S.27-1}\\
[[\delta(t_1),t_2]+[t_1,\delta(t_2)],m]=&0,\label{S.27-2}\\
[[\delta(t_1),t_2]+[t_1,\delta(t_2)],t]=&0\label{S.27-3}
\end{align}
for each $n\geq 2$ and for each $m,m_1,m_2$ in $e\mathcal{A}f$ and $t,t_1,t_2$ in $f\mathcal{A}e$.
Considering (\ref{S.26}), (\ref{S.27-1}), (\ref{S.27-2}), (\ref{S.27-3}) and Lemma \ref{S_3}, it follows that
\begin{equation}\label{S.28}
\begin{array}{lll}
[\delta(m_1),m_2]+[m_1,\delta(m_2)]\in\mathcal{Z}(\mathcal{A}) & \textrm{and} & [\delta(t_1),t_2]+[t_1,\delta(t_2)]\in\mathcal{Z}(\mathcal{A}).
\end{array}
\end{equation}
The subtraction of (\ref{S.25}) and (\ref{S.28}) leads to $2[\delta(m_1),m_2]\in\mathcal{Z}(\mathcal{A})$ and
$2[\delta(t_1),t_2]\in\mathcal{Z}(\mathcal{A})$. Since $\mathcal{A}$ is $2$-torsion free, we have that
\begin{equation}\label{S.29}
\begin{array}{lll}
[\delta(m_1),m_2]=\delta(m_1)m_2-m_2\delta(m_1)\in\mathcal{Z}(\mathcal{A}) & \textrm{and} &
[\delta(t_1),t_2]=\delta(t_1)t_2-t_2\delta(t_1)\in\mathcal{Z}(\mathcal{A}).
\end{array}
\end{equation}
Hence,
$$\begin{array}{lll}m_2\delta(m_1)\in\mathcal{Z}(e\mathcal{A}e) & \textrm{and} & \delta(t_1)t_2\in\mathcal{Z}(e\mathcal{A}e)\end{array}$$
for each $m_1,m_2$ in $e\mathcal{A}f$ and $t_1,t_2$ in $f\mathcal{A}e$. It follows obviously that $e\mathcal{A}f\cdot\delta(e\mathcal{A}f)$ and $\delta(f\mathcal{A}e)\cdot f\mathcal{A}e$ are central ideals of $e\mathcal{A}e$. Since $e\mathcal{A}e$ has no nonzero central ideal, we confirm that $e\mathcal{A}f\cdot\delta(e\mathcal{A}f)=\delta(f\mathcal{A}e)\cdot f\mathcal{A}e=\{0\}$. According to (\ref{S.29}), we confirm that $\delta(e\mathcal{A}f)\cdot e\mathcal{A}f=f\mathcal{A}e\cdot \delta(f\mathcal{A}e)=\{0\}$. By lemma \ref{S_7}, we obtain that $\delta$ is an antiderivation.\\
\indent The proof in case that $f\mathcal{A}f$ has no nonzero central ideal goes in a similar way.
\end{proof}

In \cite[Theorem 2]{Cheung-Commuting}, Cheung considers commuting mappings of triangular algebras and introduces the idea that there exists $m_0$ in $e\mathcal{A}f$ such that
$$\mathcal{Z}(\mathcal{A})=\{\begin{array}{l|l}
                            a+b & a\in \mathcal{Z}(e\mathcal{A}e),~b\in \mathcal{Z}(f\mathcal{A}f),~
                                    am_0=m_0b
                            \end{array}\},$$
which comes out to be very useful in his paper. Later in \cite[Theorem 3.6]{Xiao-Wei} and \cite[Theorem 3.4]{Li-Wei}, the authors use some similar idea that there exists $m_0$ in $e\mathcal{A}f$ and $t_0$ in $f\mathcal{A}e$ such that
$$\mathcal{Z}(\mathcal{A})=\{\begin{array}{l|l}
                            a+b & a\in \mathcal{Z}(e\mathcal{A}e),~b\in \mathcal{Z}(f\mathcal{A}f),~
                                    am_0=m_0b,~t_0a=bt_0
                            \end{array}\},$$
which also turns out to play an important role in discussing mappings on generalized matrix algebras.

\begin{lemma}\label{S_11}
Let $\varphi$ be a Lie $n$-derivation on $\mathcal{A}$. Suppose that $\mathcal{A}$ is $2$- and $(n-1)$-torsion free, and that $\mathcal{A}$ satisfies the property (\ref{1.6}). If
\begin{description}
\item[(i)] $f\varphi(e\mathcal{A}e)f\subseteq f\mathcal{Z}(\mathcal{A})f$ and $e\varphi(f\mathcal{A}f)e\subseteq e\mathcal{Z}(\mathcal{A})e$,
\end{description}
\noindent and if one of the following statements holds:
\begin{description}
\item[(ii-1)] $\mathcal{Z}(\mathcal{A})=\{\begin{array}{l|l}
                            a+b & a\in e\mathcal{Z}(\mathcal{A})e,~b\in f\mathcal{Z}(\mathcal{A})f,~
                                    am_0=m_0b
                            \end{array}\}$ for some $m_0\in e\mathcal{A}f$,
\item[(ii-2)] $\mathcal{Z}(\mathcal{A})=\{\begin{array}{l|l}
                            a+b & a\in e\mathcal{Z}(\mathcal{A})e,~b\in f\mathcal{Z}(\mathcal{A})f,~
                                    t_0a=bt_0
                            \end{array}\}$ for some $t_0\in f\mathcal{A}e$,
\end{description}
then $\varphi$ satisfies $e\varphi(tm)e+f\varphi(mt)f\in\mathcal{Z}(\mathcal{A})$ for each $m\in e\mathcal{A}f$ and $t\in f\mathcal{A}e$.
\end{lemma}

\begin{proof}
Without loss of generality, we suppose that (i) and (ii-1) hold. Discussing similarly as Theorem \ref{S_4}, we obtain same results as Claims 1, 2 and 3 in Theorem \ref{S_4}. According to (i) and the definition (\ref{S.11}) of $d$, we conclude that for each $a$ in $e\mathcal{A}e$, $m$ in $e\mathcal{A}f$, $t$ in $f\mathcal{A}e$ and $b$ in $f\mathcal{A}f$,
\begin{align}
    &d(am)=d(a)m+ad(m),\label{S.30-1}\\
    &d(mb)=md(b)+d(m)b,\label{S.30-2}\\
    &d(m)t+md(t)=d(mt)+\tau^{-1}(f\varphi(mt)f)-e\varphi(tm)e,\label{S.31-1}\\
    &d(t)m+td(m)=d(tm)+\tau(e\varphi(tm)e)-f\varphi(mt)f.\label{S.31-2}
\end{align}
For each $m$ in $e\mathcal{A}f$ and $t$ in $f\mathcal{A}e$, it follows from (\ref{S.30-1}) and (\ref{S.30-2}) that
$$\begin{array}{lll}
d(mtm)=d((mt)m)=d(mt)m+mtd(m)& \textrm{and} &d(mtm)=d(m(tm))=md(tm)+d(m)tm.
\end{array}$$
Thus,
\begin{equation}\label{S.32}
  (d(mt)-d(m)t)m=m(d(tm)-td(m)).
\end{equation}
Denote that $\epsilon(m,t)=e\varphi(tm)e-\tau^{-1}(f\varphi(mt)f)$ for $m$ in $e\mathcal{A}f$ and $t$ in $f\mathcal{A}e$.
By (\ref{S.31-1}) and (\ref{S.31-2}), we have that
$$\begin{array}{l}
    d(mt)-d(m)t=md(t)+\epsilon(m,t),\\
    d(tm)-td(m)=d(t)m-\tau(\epsilon(m,t)).
    \end{array}$$
In view of (\ref{S.32}), it follows that
$$md(t)m+\epsilon(m,t)m=md(t)m-m\tau(\epsilon(m,t)).$$
Thus, $2\epsilon(m,t)m=0$. Since $\mathcal{A}$ is $2$-torsion free, we have
\begin{equation}\label{S.33}
\epsilon(m,t)m=0
\end{equation}
for each $m$ in $e\mathcal{A}f$ and $t$ in $f\mathcal{A}e$. Let $m_0$ be as in (ii-1), then $\epsilon(m_0,t)m_0=0=m_00$. It follows from (ii-1) that $\epsilon(m_0,t)+0\in\mathcal{Z}(\mathcal{A})$. Thus, $\epsilon(m_0,t)=0$.
Since (\ref{S.33}) and
$$0=\epsilon(m+m_0,t)(m+m_0)=\epsilon(m,t)m+\epsilon(m,t)m_0+\epsilon(m_0,t)(m+m_0),$$
we conclude that $\epsilon(m,t)m_0=0=m_00$, which follows from (ii-1) that $\epsilon(m,t)+0\in\mathcal{Z}(\mathcal{A})$. Thus, $\epsilon(m,t)=0$ and $e\varphi(tm)e+f\varphi(mt)f\in\mathcal{Z}(\mathcal{A})$ for each $m$ in $e\mathcal{A}f$ and $t$ in $f\mathcal{A}e$.\\
\indent The proof in case that (i) and (ii-2) hold goes in a similar way.
\end{proof}

Let $\tilde{\mathcal{A}}$ be an arbitrary algebra. Denote $\mathcal{S}(\tilde{\mathcal{A}})$ as the subalgebra of $\tilde{\mathcal{A}}$ which is generated with all idempotents and commutators of $\tilde{\mathcal{A}}$. We would mention that if $\mathcal{A}$ is a unital algebra with a nontrivial idempotent $e$, then
$$\mathcal{S}(\mathcal{A})=(\mathcal{S}(e\mathcal{A}e)+e\mathcal{A}f\cdot f\mathcal{A}e)+e\mathcal{A}f+ f\mathcal{A}e+(\mathcal{S}(f\mathcal{A}f)+f\mathcal{A}e\cdot e\mathcal{A}f).$$
If $e\mathcal{A}f\cdot f\mathcal{A}e=f\mathcal{A}e\cdot e\mathcal{A}f=\{0\}$, we have that $\mathcal{A}=\mathcal{S}(\mathcal{A})$ if and only if $e\mathcal{A}e=\mathcal{S}(e\mathcal{A}e)$ and $f\mathcal{A}f=\mathcal{S}(f\mathcal{A}f)$. In \cite[Theorem 11]{Cheung-Lie}, Cheung arises the idea that $e\mathcal{A}e=\mathcal{S}(e\mathcal{A}e)$ or $f\mathcal{A}f=\mathcal{S}(f\mathcal{A}f)$, which turns out to play an important role in discussing sufficient conditions for Lie derivations to be standard on triangular algebras. Later in \cite[Theorem 3.4]{Mokhtari-Vishki}, the authors use this idea to consider Lie derivations on generalized matrix algebras. In the following lemma, we will use this idea to consider Lie $n$-derivations on unital algebras with nontrivial idempotents.
Besides, we also consider a very useful condition that for each $x$ in $\mathcal{A}$,
\begin{equation}\label{S.34}
  [x,\mathcal{A}]\subseteq\mathcal{Z}(\mathcal{A})\ \ \textrm{implies}\ \ x\in\mathcal{Z}(\mathcal{A}).
\end{equation}
That is, $$[[x,\mathcal{A}],\mathcal{A}]=\{0\}\ \ \textrm{implies}\ \ [x,\mathcal{A}]=\{0\},$$ which is equivalent to the condition that there exists no nonzero central inner derivation of $\mathcal{A}$. Important examples of algebras satisfying (\ref{S.34}) include commutative algebras, prime algebras, triangular algebras and matrix algebras. We would mention that if $\mathcal{A}$ is a unital algebra with a nontrivial idempotent $e$ satisfying the property (\ref{1.6}), then $\mathcal{A}$ satisfies (\ref{S.34}), but we cannot confirm that $e\mathcal{A}e$ or $f\mathcal{A}f$ satisfies (\ref{S.34}). In \cite[Theorem 5.9]{Benkovic-Eremita} and \cite[Theorem 1]{Wang-Wang}, the authors consider multiplicative Lie $n$-derivations of triangular algebras and generalized matrix algebras.

\begin{lemma}\label{S_12}
Let $\varphi$ be a Lie $n$-derivation on $\mathcal{A}$. Suppose that $\mathcal{A}$ is $2$- and $(n-1)$-torsion free, and that $\mathcal{A}$ satisfies the property (\ref{1.6}). If one of the following statements holds:
\begin{description}
\item[(i)] $e\mathcal{A}e=\mathcal{S}(e\mathcal{A}e)$ and $f\mathcal{A}f=\mathcal{S}(f\mathcal{A}f)$;
\item[(ii)] $e\mathcal{A}e=\mathcal{S}(e\mathcal{A}e)$ and $\mathcal{Z}(e\mathcal{A}e)=e\mathcal{Z}(\mathcal{A})e$;
\item[(iii)] $f\mathcal{A}f=\mathcal{S}(f\mathcal{A}f)$ and $\mathcal{Z}(f\mathcal{A}f)=f\mathcal{Z}(\mathcal{A})f$;
\item[(iv)] $e\mathcal{A}e$ or $f\mathcal{A}f$ satisfies (\ref{S.34}) when $n\geq 3$, $\mathcal{Z}(e\mathcal{A}e)=e\mathcal{Z}(\mathcal{A})e$ and $\mathcal{Z}(f\mathcal{A}f)=f\mathcal{Z}(\mathcal{A})f$;
\end{description}
then $\varphi$ satisfies $f\varphi(e\mathcal{A}e)f\subseteq f\mathcal{Z}(\mathcal{A})f$ and $e\varphi(f\mathcal{A}f)e\subseteq e\mathcal{Z}(\mathcal{A})e$.
\end{lemma}

\begin{proof}
Since $\varphi$ is a Lie $n$-derivation on $\mathcal{A}$, discussing similarly as Theorem \ref{S_4}, we obtain same results as Claims 1, 2 and 3 in Theorem \ref{S_4}. Thus for each $a$ in $e\mathcal{A}e$, $m$ in $e\mathcal{A}f$, $t$ in $f\mathcal{A}e$ and $b$ in $f\mathcal{A}f$, we have that
\begin{align}
    e\varphi(am)f=&e\varphi(a)e\cdot m+a\cdot e\varphi(m)f-m\cdot f\varphi(a)f,\label{S.35-1}\\
    e\varphi(mb)f=&m\cdot f\varphi(b)f+e\varphi(m)f\cdot b-e\varphi(b)e\cdot m,\label{S.35-2}\\
    f\varphi(ta)e=&t\cdot e\varphi(a)e+f\varphi(t)e\cdot a-f\varphi(a)f\cdot t,\label{S.35-3}\\
    f\varphi(bt)e=&f\varphi(b)f\cdot t+b\cdot f\varphi(t)e-t\cdot e\varphi(b)e.\label{S.35-4}
\end{align}
Then the remaining proof could be organized by the following claims.\\

\noindent$\begin{array}{ll}\textbf{Claim 1.}
            & e\mathcal{A}e=\mathcal{S}(e\mathcal{A}e)~\textrm{implies}~f\varphi(e\mathcal{A}e)f\subseteq f\mathcal{Z}(\mathcal{A})f.
            \end{array}$\\
\indent Let $A_0=\{\begin{array}{l|l}
a\in e\mathcal{A}e & f\varphi(a)f\in f\mathcal{Z}(\mathcal{A})f
\end{array}\}$. We only need to prove that $A_0=e\mathcal{A}e$.\\
\indent According to the linearity of $\varphi$, we have that $A_0$ is a $\mathcal{R}$-submodule of $e\mathcal{A}e$. Take arbitrary elements $a,a'$ in $A_0$. We have $f\varphi(a)f, f\varphi(a')f\in f\mathcal{Z}(\mathcal{A})f$. By Lemma \ref{S_2}, we have that for each $m$ in $e\mathcal{A}f$ and $t$ in $f\mathcal{A}e$,
\begin{align*}
    m\cdot f\varphi(a)f=&\tau^{-1}(f\varphi(a)f)\cdot m,\\
    m\cdot f\varphi(a')f=&\tau^{-1}(f\varphi(a')f)\cdot m,\\
    f\varphi(a)f\cdot t=&t\cdot\tau^{-1}(f\varphi(a)f),\\
    f\varphi(a')f\cdot t=&t\cdot\tau^{-1}(f\varphi(a')f).
\end{align*}
It follows from (\ref{S.35-1}) and (\ref{S.35-3}) that
\begin{align*}
  e\varphi((aa')m)f=&e\varphi(aa')e\cdot m+aa'\cdot e\varphi(m)f-m\cdot f\varphi(aa')f,\\
  e\varphi(a(a'm))f=&e\varphi(a)e\cdot a'm+a\cdot e\varphi(a'm)f-a'm\cdot f\varphi(a)f\\
  =&e\varphi(a)e\cdot a'm+a\cdot e\varphi(a')e\cdot m+aa'\cdot e\varphi(m)f-am\cdot f\varphi(a')f-a'm\cdot f\varphi(a)f\\
  =&(e\varphi(a)e\cdot a'+a\cdot e\varphi(a')e-a\tau^{-1}(f\varphi(a')f)-a'\tau^{-1}(f\varphi(a)f))\cdot m+aa'\cdot e\varphi(m)f,\\
  f\varphi(t(aa'))e=&t\cdot e\varphi(aa')e+f\varphi(t)e\cdot aa'-f\varphi(aa')f\cdot t,\\
  f\varphi((ta)a')e=&ta\cdot e\varphi(a')e+f\varphi(ta)e\cdot a'-f\varphi(a')f\cdot ta\\
  =&ta\cdot e\varphi(a')e+t\cdot e\varphi(a)e\cdot a'+f\varphi(t)e\cdot aa'-f\varphi(a)f\cdot ta'-f\varphi(a')f\cdot ta\\
  =&t\cdot(a\cdot e\varphi(a')e+e\varphi(a)e\cdot a'-a'\tau^{-1}(f\varphi(a)f)-a\tau^{-1}(f\varphi(a')f))
  +f\varphi(t)e\cdot aa'.
\end{align*}
\noindent Then
\begin{align*}
  &(e\varphi(aa')e-e\varphi(a)e\cdot a'-a\cdot e\varphi(a')e+a\tau^{-1}(f\varphi(a')f)+a'\tau^{-1}(f\varphi(a)f))\cdot m=m\cdot f\varphi(aa')f,\\
  &f\varphi(aa')f\cdot t=t\cdot(e\varphi(aa')e-e\varphi(a)e\cdot a'-a\cdot e\varphi(a')e+a\tau^{-1}(f\varphi(a')f)+a'\tau^{-1}(f\varphi(a)f)).
\end{align*}
\noindent According to Lemma \ref{S_2}, $f\varphi(aa')f\in f\mathcal{Z}(\mathcal{A})f$. That is, $aa'\in A_0$. Thus, $A_0$ is a subalgebra of $e\mathcal{A}e$.\\
\indent Take an arbitrary element $a$ in $e\mathcal{A}e$ satisfying $a=a^2$. By (\ref{S.35-1}) and (\ref{S.35-3}), we have that for each $m$ in $e\mathcal{A}f$ and $t$ in $f\mathcal{A}e$,
\begin{align*}
  e\varphi(am)f=&e\varphi(a(am))f=e\varphi(a)e\cdot am+a\cdot e\varphi(am)f-am\cdot f\varphi(a)f\\
  =&e\varphi(a)e\cdot am+a\cdot e\varphi(a)e\cdot m+a\cdot e\varphi(m)f-2am\cdot f\varphi(a)f,\\
  f\varphi(ta)e=&f\varphi((ta)a)e=ta\cdot e\varphi(a)e+f\varphi(ta)e\cdot a-f\varphi(a)f\cdot ta\\
  =&ta\cdot e\varphi(a)e+t\cdot e\varphi(a)e\cdot a+f\varphi(t)e\cdot a-2f\varphi(a)f\cdot ta.
\end{align*}
\noindent Then combining with (\ref{S.35-1}) and (\ref{S.35-3}), we obtain that
\begin{align}
    &(e\varphi(a)e-e\varphi(a)e\cdot a-a\cdot e\varphi(a)e)\cdot m+2am\cdot f\varphi(a)f=m\cdot f\varphi(a)f,\label{S.36-1}\\
    &t\cdot(e\varphi(a)e-a\cdot e\varphi(a)e-e\varphi(a)e\cdot a)+2f\varphi(a)f\cdot ta=f\varphi(a)f\cdot t. \label{S.36-2}
\end{align}
Left and right multiplication by $a$ respectively implies that
\begin{align}
    &am\cdot f\varphi(a)f=a\cdot e\varphi(a)e\cdot am,\label{S.37-1}\\
    &f\varphi(a)f\cdot ta=ta\cdot e\varphi(a)e\cdot a. \label{S.37-2}
\end{align}
In view of (\ref{S.37-1}) and (\ref{S.37-2}), equations (\ref{S.36-1}) and (\ref{S.36-2}) can be reformed to
\begin{align*}
    &(e\varphi(a)e-e\varphi(a)e\cdot a-a\cdot e\varphi(a)e+2a\cdot e\varphi(a)e\cdot a)\cdot m=m\cdot f\varphi(a)f,\\
    &t\cdot(e\varphi(a)e-e\varphi(a)e\cdot a-a\cdot e\varphi(a)e+2a\cdot e\varphi(a)e\cdot a)=f\varphi(a)f\cdot t.
\end{align*}
\noindent According to Lemma \ref{S_2}, $f\varphi(a)f\in f\mathcal{Z}(\mathcal{A})f$. That is, $a=a^2\in A_0$. Thus, $A_0$ contains all idempotents in $e\mathcal{A}e$.\\
\indent Take arbitrary elements $a,a'$ in $e\mathcal{A}e$. By (\ref{S.35-1}) and (\ref{S.35-3}), we have that for each $m$ in $e\mathcal{A}f$ and $t$ in $f\mathcal{A}e$,
\begin{align*}
  e\varphi([a,a']m)f=&e\varphi([a,a'])e\cdot m+[a,a']\cdot e\varphi(m)f-m\cdot f\varphi([a,a'])f,\\
  e\varphi([a,a']m)f=&e\varphi(aa'm)f-e\varphi(a'am)f\\
  =&[e\varphi(a)e,a']m+[a,e\varphi(a')e]\cdot m+[a,a']\cdot e\varphi(m)f,\\
  f\varphi(t[a,a'])e=&t\cdot e\varphi([a,a'])e+f\varphi(t)e\cdot [a,a']-f\varphi([a,a'])f\cdot t,\\
  f\varphi(t[a,a'])e=&f\varphi(taa')e-f\varphi(ta'a)e\\
  =&t\cdot[a,e\varphi(a')e]+t\cdot[e\varphi(a)e,a']+f\varphi(t)e\cdot[a,a'].
\end{align*}
\noindent Then
\begin{align*}
    &(e\varphi([a,a'])e-[e\varphi(a)e,a']-[a,e\varphi(a')e])\cdot m=m\cdot f\varphi([a,a'])f,\\
    &f\varphi([a,a'])f\cdot t=t\cdot (e\varphi([a,a'])e-[a,e\varphi(a')e]-[e\varphi(a)e,a']).
\end{align*}
\noindent According to Lemma \ref{S_2}, $f\varphi([a,a'])f\in f\mathcal{Z}(\mathcal{A})f$. That is, $[a,a']\in A_0$. Thus, $A_0$ contains all commutators in $e\mathcal{A}e$.\\
\indent Since $A_0$ is a subalgebra of $e\mathcal{A}e$, and $A_0$ contains all idempotent and commutators in $e\mathcal{A}e$, we have that $\mathcal{S}(e\mathcal{A}e)\subseteq A_0$. Since $e\mathcal{A}e=\mathcal{S}(e\mathcal{A}e)$, we conclude that $e\mathcal{A}e=A_0$.\\

\noindent$\begin{array}{ll}\textbf{Claim 2.}
            & f\mathcal{A}f=\mathcal{S}(f\mathcal{A}f)~\textrm{implies}~e\varphi(f\mathcal{A}f)e\subseteq e\mathcal{Z}(\mathcal{A})e.
            \end{array}$\\
\indent The proof of Claim 2 is similar to the proof of Claim 1.\\

\noindent$\begin{array}{ll}\textbf{Claim 3.}
            & \textrm{(i) implies that}~f\varphi(e\mathcal{A}e)f\subseteq f\mathcal{Z}(\mathcal{A})f~\textrm{and}~e\varphi(f\mathcal{A}f)e\subseteq e\mathcal{Z}(\mathcal{A})e.
            \end{array}$\\
\indent It's obvious according to Claim 1 and 2.\\

\noindent$\begin{array}{ll}\textbf{Claim 4.}
            & \textrm{(ii) implies that}~f\varphi(e\mathcal{A}e)f\subseteq f\mathcal{Z}(\mathcal{A})f~\textrm{and}~e\varphi(f\mathcal{A}f)e\subseteq e\mathcal{Z}(\mathcal{A})e.
            \end{array}$\\
\indent By Claim 1, we only need to prove $e\varphi(f\mathcal{A}f)e\subseteq e\mathcal{Z}(\mathcal{A})e$. If $n\geq 3$, for each $a$ in $e\mathcal{A}e$, $m$ in $e\mathcal{A}f$, $t$ in $f\mathcal{A}e$ and $b$ in $f\mathcal{A}f$, since $[a,b]=0$ and $p_n(a,b,m,f,...,f)=p_n(a,b,t,e,...,e)=0$, we obtain that
\begin{align*}
     0=&\varphi(p_n(a,b,m,f,...,f))=p_n(\varphi(a),b,m,f,...,f)+p_n(a,\varphi(b),m,f,...,f)\\
     =&[[\varphi(a),b]+[a,\varphi(b)],m],\\
     0=&\varphi(p_n(a,b,t,e,...,e))=p_n(\varphi(a),b,t,e,...,e)+p_n(a,\varphi(b),t,e,...,e)\\
     =&[[\varphi(a),b]+[a,\varphi(b)],t].
\end{align*}
According to Lemma \ref{S_2}, we have $$[f\varphi(a)f,b]+[a,e\varphi(b)e]\in\mathcal{Z}(\mathcal{A}).$$
By Claim 1, we have
$f\varphi(e\mathcal{A}e)f\subseteq f\mathcal{Z}(\mathcal{A})f\subseteq\mathcal{Z}(f\mathcal{A}f)$, and $[f\varphi(a)f,b]=0$. If $n\geq 3$, it follows that $[a,e\varphi(b)e]=0$ for each $a$ in $e\mathcal{A}e$ and $b$ in $f\mathcal{A}f$. If $n=2$, for each $a$ in $e\mathcal{A}e$ and $b$ in $f\mathcal{A}f$, we have that
$$0=\varphi([a,b])=[f\varphi(a)f,b]+[a,e\varphi(b)e]$$
which follows that $[a,e\varphi(b)e]=0$. Thus, we have $e\varphi(f\mathcal{A}f)e\subseteq\mathcal{Z}(e\mathcal{A}e)$ for each $n\geq 2$. Since $\mathcal{Z}(e\mathcal{A}e)=e\mathcal{Z}(\mathcal{A})e$, we conclude that $e\varphi(f\mathcal{A}f)e\subseteq e\mathcal{Z}(\mathcal{A})e$.\\

\noindent$\begin{array}{ll}\textbf{Claim 5.}
            & \textrm{(iii) implies that}~f\varphi(e\mathcal{A}e)f\subseteq f\mathcal{Z}(\mathcal{A})f~\textrm{and}~e\varphi(f\mathcal{A}f)e\subseteq e\mathcal{Z}(\mathcal{A})e.
            \end{array}$\\
\indent The proof of Claim 5 is similar to the proof of Claim 4.\\

\noindent$\begin{array}{ll}\textbf{Claim 6.}
            & \textrm{(iv) implies that}~f\varphi(e\mathcal{A}e)f\subseteq f\mathcal{Z}(\mathcal{A})f~\textrm{and}~e\varphi(f\mathcal{A}f)e\subseteq e\mathcal{Z}(\mathcal{A})e.
            \end{array}$\\
\indent If $n\geq 3$, suppose that $e\mathcal{A}e$ satisfies (\ref{S.34}). Similar to the proof of Claim 4, we have that
\begin{equation}\label{S.38}
[f\varphi(a)f,b]+[a,e\varphi(b)e]\in\mathcal{Z}(\mathcal{A})
\end{equation}
for each $a$ in $e\mathcal{A}e$ and $b$ in $f\mathcal{A}f$. Then
$[e\mathcal{A}e,e\varphi(f\mathcal{A}f)e]\subseteq e\mathcal{Z}(\mathcal{A})e\subseteq \mathcal{Z}(e\mathcal{A}e)$.
Since $e\mathcal{A}e$ satisfies (\ref{S.34}), it follows that
\begin{equation}\label{S.39}
e\varphi(f\mathcal{A}f)e\subseteq\mathcal{Z}(e\mathcal{A}e).
\end{equation}
Since $\mathcal{Z}(e\mathcal{A}e)=e\mathcal{Z}(\mathcal{A})e$, we have that $e\varphi(f\mathcal{A}f)e\subseteq e\mathcal{Z}(\mathcal{A})e$. On the other hand, by (\ref{S.39}), we have $[a,e\varphi(b)e]=0$ for each $a$ in $e\mathcal{A}e$ and $b$ in $f\mathcal{A}f$. In view of (\ref{S.38}), we obtain $[f\varphi(a)f,b]=0$ for each $a$ in $e\mathcal{A}e$ and $b$ in $f\mathcal{A}f$. That is, $f\varphi(e\mathcal{A}e)f\subseteq \mathcal{Z}(f\mathcal{A}f)$. Since $\mathcal{Z}(f\mathcal{A}f)=f\mathcal{Z}(\mathcal{A})f$, we obtain that $f\varphi(e\mathcal{A}e)f\subseteq f\mathcal{Z}(\mathcal{A})f$. The proof in case that $f\mathcal{A}f$ satisfies (\ref{S.34}) goes in a similar way.\\
\indent If $n=2$, for each $a$ in $e\mathcal{A}e$ and $b$ in $f\mathcal{A}f$, we have that
$$0=\varphi([a,b])=[f\varphi(a)f,b]+[a,e\varphi(b)e].$$
Then, $[f\varphi(a)f,b]=[a,e\varphi(b)e]=0$. That is, $f\varphi(e\mathcal{A}e)f\subseteq \mathcal{Z}(f\mathcal{A}f)$ and
$e\varphi(f\mathcal{A}f)e\subseteq\mathcal{Z}(e\mathcal{A}e)$. Since $\mathcal{Z}(f\mathcal{A}f)=f\mathcal{Z}(\mathcal{A})f$ and $\mathcal{Z}(e\mathcal{A}e)=e\mathcal{Z}(\mathcal{A})e$, we conclude that $f\varphi(e\mathcal{A}e)f\subseteq f\mathcal{Z}(\mathcal{A})f$ and
$e\varphi(f\mathcal{A}f)e\subseteq e\mathcal{Z}(\mathcal{A})e$.
\end{proof}

Associating previous results, we conclude the following corollary.

\begin{corollary}\label{S_13}
Let $\varphi$ be a Lie $n$-derivation on $\mathcal{A}$. Suppose that $\mathcal{A}$ is $2$- and $(n-1)$-torsion free, and that $\mathcal{A}$ satisfies the property (\ref{1.6}). If one of the following statements holds:
\begin{description}
\item[(i-1)] $f\varphi(e\mathcal{A}e)f\subseteq f\mathcal{Z}(\mathcal{A})f$ and $e\varphi(f\mathcal{A}f)e\subseteq e\mathcal{Z}(\mathcal{A})e$,
\item[(i-2)] $e\mathcal{A}e=\mathcal{S}(e\mathcal{A}e)$ and $f\mathcal{A}f=\mathcal{S}(f\mathcal{A}f)$,
\item[(i-3)] $e\mathcal{A}e=\mathcal{S}(e\mathcal{A}e)$ and $\mathcal{Z}(e\mathcal{A}e)=e\mathcal{Z}(\mathcal{A})e$,
\item[(i-4)] $f\mathcal{A}f=\mathcal{S}(f\mathcal{A}f)$ and $\mathcal{Z}(f\mathcal{A}f)=f\mathcal{Z}(\mathcal{A})f$,
\item[(i-5)] $e\mathcal{A}e$ or $f\mathcal{A}f$ satisfies (\ref{S.34}) when $n\geq 3$, $\mathcal{Z}(e\mathcal{A}e)=e\mathcal{Z}(\mathcal{A})e$ and $\mathcal{Z}(f\mathcal{A}f)=f\mathcal{Z}(\mathcal{A})f$,
\end{description}
and if one of the following statements also holds:
\begin{description}
\item[(ii-1)] $e\varphi(tm)e+f\varphi(mt)f\in\mathcal{Z}(\mathcal{A})$ for each $m\in e\mathcal{A}f$ and $t\in f\mathcal{A}e$.
\item[(ii-2)] $e\mathcal{A}e$ or $f\mathcal{A}f$ has no nonzero central ideal,
\item[(ii-3)] $\mathcal{Z}(\mathcal{A})=\{\begin{array}{l|l}
                            a+b & a\in e\mathcal{Z}(\mathcal{A})e,~b\in f\mathcal{Z}(\mathcal{A})f,~
                                    am_0=m_0b
                            \end{array}\}$ for some $m_0\in e\mathcal{A}f$,
\item[(ii-4)] $\mathcal{Z}(\mathcal{A})=\{\begin{array}{l|l}
                            a+b & a\in e\mathcal{Z}(\mathcal{A})e,~b\in f\mathcal{Z}(\mathcal{A})f,~
                                    t_0a=bt_0
                            \end{array}\}$ for some $t_0\in f\mathcal{A}e$,
\item[(ii-5)] $\mathcal{A}$ satisfies (\ref{S.22-2}),
\end{description}
then $\varphi$ is of the form $\varphi=d+\delta+\gamma$, where $d$ is a derivation on $\mathcal{A}$, $\delta$ is a singular Jordan derivation on $\mathcal{A}$, and $\gamma$ is a linear mapping from $\mathcal{A}$ into $\mathcal{Z}(\mathcal{A})$ vanishing on all $(n-1)-$th commutators of $\mathcal{A}$. \\
\indent In addition, $\delta$ is also an antiderivation on $\mathcal{A}$ when $\mathrm{(ii}$-$\mathrm{2)}$ or $\mathrm{(ii}$-$\mathrm{5)}$ holds.\\
\indent Furthermore, let us make a further assumption that $n$ is even, or that $f\varphi(e\mathcal{A}f)e=e\varphi(f\mathcal{A}e)f=\{0\}$. Then $\varphi$ is standard.
\end{corollary}

\begin{remark}
Corollary \ref{S_13} improves \cite[Theorem 2.1]{Wang} and \cite[Theorem 5.1]{Benkovic} which consider Lie triple derivations and Lie $n$-derivations on a unital algebra with a nontrivial idempotent $e$ satisfying the property (\ref{1.6}).
\end{remark}

Now, we obtain the sufficient conditions under which every Lie $n$-derivation can be standard.

\begin{corollary}\label{S_14}
Suppose that $n$ is even, and that $\mathcal{A}$ is a $2$- and $(n-1)$-torsion free algebra satisfying the property (\ref{1.6}). If one of the following statements holds:
\begin{description}
\item[(i-1)] $e\mathcal{A}e=\mathcal{S}(e\mathcal{A}e)$ and $f\mathcal{A}f=\mathcal{S}(f\mathcal{A}f)$,
\item[(i-2)] $e\mathcal{A}e=\mathcal{S}(e\mathcal{A}e)$ and $\mathcal{Z}(e\mathcal{A}e)=e\mathcal{Z}(\mathcal{A})e$,
\item[(i-3)] $f\mathcal{A}f=\mathcal{S}(f\mathcal{A}f)$ and $\mathcal{Z}(f\mathcal{A}f)=f\mathcal{Z}(\mathcal{A})f$,
\item[(i-4)] $e\mathcal{A}e$ or $f\mathcal{A}f$ satisfies (\ref{S.34}) when $n\geq 3$, $\mathcal{Z}(e\mathcal{A}e)=e\mathcal{Z}(\mathcal{A})e$ and $\mathcal{Z}(f\mathcal{A}f)=f\mathcal{Z}(\mathcal{A})f$,
\end{description}
and if one of the following statements also holds:
\begin{description}
\item[(ii-1)] $e\mathcal{A}e$ or $f\mathcal{A}f$ has no nonzero central ideal,
\item[(ii-2)] $\mathcal{Z}(\mathcal{A})=\{\begin{array}{l|l}
                            a+b & a\in e\mathcal{Z}(\mathcal{A})e,~b\in f\mathcal{Z}(\mathcal{A})f,~
                                    am_0=m_0b
                            \end{array}\}$ for some $m_0\in e\mathcal{A}f$,
\item[(ii-3)] $\mathcal{Z}(\mathcal{A})=\{\begin{array}{l|l}
                            a+b & a\in e\mathcal{Z}(\mathcal{A})e,~b\in f\mathcal{Z}(\mathcal{A})f,~
                                    t_0a=bt_0
                            \end{array}\}$ for some $t_0\in f\mathcal{A}e$,
\item[(ii-4)] $\mathcal{A}$ satisfies (\ref{S.22-2}),
\end{description}
then every Lie $n$-derivation on $\mathcal{A}$ is standard.
\end{corollary}

\begin{corollary}\label{S_15}
Suppose that $\mathcal{A}$ is $2$- and $(n-1)$-torsion free, and that $\mathcal{A}$ satisfies the property (\ref{1.6}). If one of the following statements holds:
\begin{description}
\item[(i-1)] $e\mathcal{A}e=\mathcal{S}(e\mathcal{A}e)$ and $f\mathcal{A}f=\mathcal{S}(f\mathcal{A}f)$,
\item[(i-2)] $e\mathcal{A}e=\mathcal{S}(e\mathcal{A}e)$ and $\mathcal{Z}(e\mathcal{A}e)=e\mathcal{Z}(\mathcal{A})e$,
\item[(i-3)] $f\mathcal{A}f=\mathcal{S}(f\mathcal{A}f)$ and $\mathcal{Z}(f\mathcal{A}f)=f\mathcal{Z}(\mathcal{A})f$,
\item[(i-4)] $e\mathcal{A}e$ or $f\mathcal{A}f$ satisfies (\ref{S.34}) when $n\geq 3$, $\mathcal{Z}(e\mathcal{A}e)=e\mathcal{Z}(\mathcal{A})e$ and $\mathcal{Z}(f\mathcal{A}f)=f\mathcal{Z}(\mathcal{A})f$,
\end{description}
and if one of the following statements also holds:
\begin{description}
\item[(ii-1)] $e\mathcal{A}e$ or $f\mathcal{A}f$ has no nonzero central ideal,
\item[(ii-2)] $\mathcal{A}$ satisfies (\ref{S.22-2}),
\end{description}
and if
\begin{description}
\item[(iii)] for each $x$ in $\mathcal{A}$, we have $exf\cdot f\mathcal{A}e=\{0\}=f\mathcal{A}e\cdot exf$ implies $exf=0$ and $e\mathcal{A}f\cdot fxe=\{0\}=fxe\cdot e\mathcal{A}f$ implies $fxe=0$,
\end{description}
then every Lie $n$-derivation on $\mathcal{A}$ is standard.
\end{corollary}

\begin{proof}
According to Corollary \ref{S_13}, we have that if $\mathcal{A}$ satisfies one of (i-1), (i-2), (i-3) and (i-4) and satisfies (ii-1) or (ii-2), then arbitrary Lie $n$-derivation $\varphi$ on $\mathcal{A}$ is of the form $\varphi=d+\delta+\gamma$, where $d$ is a derivation on $\mathcal{A}$, $\delta$ is a singular Jordan derivation and antiderivation on $\mathcal{A}$, and $\gamma$ is a linear mapping from $\mathcal{A}$ into $\mathcal{Z}(\mathcal{A})$ vanishing on all $(n-1)-$th commutators of $\mathcal{A}$.
By Remark \ref{S_7}, we know that $\delta$ satisfies (\ref{S.22-1}):
$$\delta(e\mathcal{A}f)\cdot e\mathcal{A}f=e\mathcal{A}f\cdot\delta(e\mathcal{A}f)=\delta(f\mathcal{A}e)\cdot f\mathcal{A}e=f\mathcal{A}e\cdot\delta(f\mathcal{A}e)=\{0\}.$$
Since (iii), we conclude that $\delta(e\mathcal{A}f)=\delta(f\mathcal{A}e)=\{0\}$. That is, $\delta=0$.
\end{proof}

We would mention that if both conditions (ii-2) and (iii) hold, it actually means that $\mathcal{A}$ satisfies $e\mathcal{A}f=f\mathcal{A}e=\{0\}$. According to the discussion before Lemma \ref{S_12}, we obtain the following corollary.

\begin{corollary}\label{S_16}
Suppose that $\mathcal{A}$ is a $2$- and $(n-1)$-torsion free algebra satisfying the property (\ref{1.6}), and that $\mathcal{A}=\mathcal{S}(\mathcal{A})$. If one of the following statements holds:
\begin{description}
\item[(i)] $n$ is even, and $\mathcal{A}$ satisfies (\ref{S.22-2}),
\item[(ii)] $e\mathcal{A}f=f\mathcal{A}e=\{0\}$,
\end{description}
then every Lie $n$-derivation on $\mathcal{A}$ is standard.
\end{corollary}

\section{\bf Applications}\label{A}

\subsection{\bf Generalized matrix algebras}
Let $\mathcal{G}=(A, M, N, B)$ be a generalized matrix algebra, where $A$ and $B$ are two unital algebras, and $_AM_B$ and $_BN_A$ are two bimodules. Suppose that $M$ is faithful, which means that $aM=0$ implies $a=0$ for each $a\in A$ and that $Mb=0$ implies $b=0$ for each $b\in B$. It obviously follows that $\mathcal{G}$ satisfies (\ref{1.6}). By Corollaries \ref{S_14} and \ref{S_15}, we obtain the following corollaries which are also partially proved by \cite[Theorem 1]{Wang-Wang}.

\begin{corollary}\label{A_1}
Let $\mathcal{G}=(A, M, N, B)$ be a $2$- and $(n-1)$-torsion free generalized matrix algebra, where $A$ and $B$ are two unital algebras, $M$ is a faithful $(A,B)$-bimodule, and $N$ is a $(B,A)$-bimodule. Suppose that $n$ is even. If one of the following statements holds:
\begin{description}
  \item[(i-1)] $A=\mathcal{S}(A)$ and $B=\mathcal{S}(B)$, 
  \item[(i-2)] $A=\mathcal{S}(A)$ and $\mathcal{Z}(A)=e\mathcal{Z}(\mathcal{G})e$,
  \item[(i-3)] $B=\mathcal{S}(B)$ and $\mathcal{Z}(B)=f\mathcal{Z}(\mathcal{G})f$,
  \item[(i-4)] $A$ or $B$ satisfies (\ref{S.34}) when $n\geq 3$, $\mathcal{Z}(A)=e\mathcal{Z}(\mathcal{G})e$ and $\mathcal{Z}(B)=f\mathcal{Z}(\mathcal{G})f$,
\end{description}
and if one of the following statements also holds:
\begin{description}
\item[(ii-1)] $A$ or $B$ has no nonzero central ideal,
\item[(ii-2)] $\mathcal{Z}(\mathcal{G})=\{\begin{array}{l|l}
                            a+b & a\in e\mathcal{Z}(\mathcal{G})e,~b\in f\mathcal{Z}(\mathcal{G})f,~
                                    am_0=m_0b
                            \end{array}\}$ for some $m_0\in M$,
\item[(ii-3)] $\mathcal{Z}(\mathcal{G})=\{\begin{array}{l|l}
                            a+b & a\in e\mathcal{Z}(\mathcal{G})e,~b\in f\mathcal{Z}(\mathcal{G})f,~
                                    t_0a=bt_0
                            \end{array}\}$ for some $t_0\in N$,
\item[(ii-4)] $\mathcal{G}$ satisfies (\ref{S.22-2}),
\end{description}
then every Lie $n$-derivation on $\mathcal{G}$ is standard.
\end{corollary}

\begin{corollary}\label{A_2}
Let $\mathcal{G}=(A, M, N, B)$ be a $2$- and $(n-1)$-torsion free generalized matrix algebra, where $A$ and $B$ are two unital algebras, $M$ is a faithful $(A,B)$-bimodule, and $N$ is a $(B,A)$-bimodule. If one of the following statements holds:
\begin{description}
\item[(i-1)] $A=\mathcal{S}(A)$ and $B=\mathcal{S}(B)$,
\item[(i-2)] $A=\mathcal{S}(A)$ and $\mathcal{Z}(A)=e\mathcal{Z}(\mathcal{G})e$,
\item[(i-3)] $B=\mathcal{S}(B)$ and $\mathcal{Z}(B)=f\mathcal{Z}(\mathcal{G})f$,
\item[(i-4)] $A$ or $B$ satisfies (\ref{S.34}) when $n\geq 3$, $\mathcal{Z}(A)=e\mathcal{Z}(\mathcal{G})e$ and $\mathcal{Z}(B)=f\mathcal{Z}(\mathcal{G})f$,
\end{description}
and if one of the following statements also holds:
\begin{description}
\item[(ii-1)] $A$ or $B$ has no nonzero central ideal,
\item[(ii-2)] $\mathcal{G}$ satisfies (\ref{S.22-2}),
\end{description}
and if
\begin{description}
\item[(iii)]  for each $m$ in $M$ and $t$ in $N$, it follows that $m\cdot N=\{0\}=N\cdot m$ implies $m=0$, and that $M\cdot t=\{0\}=t\cdot M$ implies $t=0$,
\end{description}
then every Lie $n$-derivation on $\mathcal{G}$ is standard.
\end{corollary}

Let's recall that a generalized matrix algebra $\mathcal{G}=(A, M, N, B)$ is called a \emph{trivial generalized matrix algebra} if $MN=NM=\{0\}$. Thus, $\mathcal{G}$ satisfies (\ref{S.22-2}). In fact, a trivial generalized matrix algebra satisfying (iii) is a generalized matrix algebra satisfying $M=N=\{0\}$. Then we obtain the following corollary.

\begin{corollary}\label{A_1_2_trivial}
Let $\mathcal{G}=(A, M, N, B)$ be a $2$- and $(n-1)$-torsion free trivial generalized matrix algebra, where $A$ and $B$ are two unital algebras, $M$ is a faithful $(A,B)$-bimodule, and $N$ is a $(B,A)$-bimodule. If one of the following statements holds:
\begin{description}
\item[(i)] $A=\mathcal{S}(A)$ and $B=\mathcal{S}(B)$,
\item[(ii)] $\mathcal{G}=\mathcal{S}(\mathcal{G})$,
\item[(iii)] $A=\mathcal{S}(A)$ and $\mathcal{Z}(A)=e\mathcal{Z}(\mathcal{G})e$,
\item[(iv)] $B=\mathcal{S}(B)$ and $\mathcal{Z}(B)=f\mathcal{Z}(\mathcal{G})f$,
\item[(v)] $A$ or $B$ satisfies (\ref{S.34}) when $n\geq 3$, $\mathcal{Z}(A)=e\mathcal{Z}(\mathcal{G})e$ and $\mathcal{Z}(B)=f\mathcal{Z}(\mathcal{G})f$,
\end{description}
and if $n$ is even or $M=N=\{0\}$, then every Lie $n$-derivation on $\mathcal{G}$ is standard.
\end{corollary}

Let $\mathcal{A}=M_s(A)$ be a full matrix algebra, where $A$ is a unital algebra with center $\mathcal{Z}(A)$ and $s\geq 2$ is an arbitrary positive integer. According to \cite[Corollary 4.4]{Benkovic-Sirovnik}, we know that every Jordan derivation of $M_s(A)$ is a derivation. If $s\geq 3$, $M_s(A)$ can be represented as a generalized matrix algebra of the form $$\left(\begin{array}{ll}
 A & M_{1\times(s-1)}(A)\\
  M_{(s-1)\times 1}(A) & M_{s-1}(A)
\end{array}\right)$$
where $M_{1\times(s-1)}(A)$ is faithful. In view of \cite[Lemma 1]{Du-Wang-lieder} and \cite[Example 5.6]{Benkovic-Eremita}, we obtain that (i-5) and (ii-2) in Corollary \ref{S_13} hold. Thus, every Lie $n$-derivation on $M_s(A)$ is standard for each $s\geq 3$. If $s=2$ and $A$ is commutative, we can assert that (i-5) and (ii-3) in Corollary \ref{S_13} hold. Actually, take arbitrary elements $a$ in $e\mathcal{Z}(M_2(A))e$ and $b$ in $f\mathcal{Z}(M_2(A))f$, which follows that $am=m\tau(a)$ and $bt=t\tau^{-1}(b)$ for each $m$ in $e\mathcal{Z}(M_2(A))f$ and $t$ in $f\mathcal{Z}(M_2(A))e$. Assume that there exists nonzero element $m_0$ in $eM_2(A)f$ satisfying $am_0=m_0b$. We can obtain that $b=\tau(a)$. Thus, $am=mb$ and $bt=ta$ for each $m$ in $e\mathcal{Z}(M_2(A))f$ and $t$ in $f\mathcal{Z}(M_2(A))e$ if $am_0=m_0b$, i.e., (ii-3) in Corollary \ref{S_13} holds. If $s=2$ and $A$ is noncommutative prime, it's not difficult to prove that (i-5) and (ii-2) in Corollary \ref{S_13} hold. Therefore, every Lie $n$-derivation on $M_2(A)$ is standard. These conclusions are also partially proved by \cite[Theorem 2.1]{Xiao-Wei-Lie}, \cite[Corollaries 5.5 and 5.6]{Benkovic} and \cite[Corollaries 2 and 3]{Wang-Wang}.

\begin{corollary}\label{A_3}
Let $\mathcal{A}=M_s(A)$ be a $2$- and $(n-1)$-torsion free full matrix algebra, where $A$ is a unital algebra with center $\mathcal{Z}(A)$ and $s\geq 3$. Then every Lie $n$-derivation on $M_s(A)$ is standard.
\end{corollary}

\begin{corollary}\label{A_4}
Let $\mathcal{A}=M_2(A)$ be a $2$- and $(n-1)$-torsion free full matrix algebra, where $A$ is a unital algebra with center $\mathcal{Z}(A)$. Suppose that $A$ is a commutative algebra or a noncommutative prime algebra. Then every Lie $n$-derivation on $M_2(A)$ is standard.
\end{corollary}

\subsection{\bf Triangular algebras}
Let $\mathcal{T}=(A, M, B)$ be a triangular algebra, where $A$ and $B$ are two unital algebras, and $M$ is a faithful $(A,B)$-bimodule. Since $f\mathcal{T}e=\{0\}$, it's obvious that $\mathcal{T}$ satisfies (\ref{1.6}) and (\ref{S.22-2}), and $f\varphi(e\mathcal{A}f)e=e\varphi(f\mathcal{A}e)f=\{0\}$. According to Corollary \ref{S_13} and the discussion before Lemma \ref{S_12}, we have the following corollary which is also partially proved by \cite[Theorem 5.9]{Benkovic-Eremita}.

\begin{corollary}\label{A_5}
Let $\mathcal{T}=(A,M,B)$ be a $2$- and $(n-1)$-torsion free triangular algebra, where $A$ and $B$ are two unital algebras, and $M$ is a faithful $(A,B)$-bimodule. If one of the following statements holds:
\begin{description}
\item[(i)] $A=\mathcal{S}(A)$ and $B=\mathcal{S}(B)$,
\item[(ii)] $\mathcal{T}=\mathcal{S}(\mathcal{T})$,
\item[(iii)] $A=\mathcal{S}(A)$ and $\mathcal{Z}(A)=e\mathcal{Z}(\mathcal{T})e$,
\item[(iv)] $B=\mathcal{S}(B)$ and $\mathcal{Z}(B)=f\mathcal{Z}(\mathcal{T})f$,
\item[(v)] $A$ or $B$ satisfies (\ref{S.34}) when $n\geq 3$, $\mathcal{Z}(A)=e\mathcal{Z}(\mathcal{T})e$ and $\mathcal{Z}(B)=f\mathcal{Z}(\mathcal{T})f$,
\end{description}
then every Lie $n$-derivation on $\mathcal{T}$ is standard.
\end{corollary}

Similar to the discussion of full matrix algebras, we conclude following corollaries.

\begin{corollary}\label{A_6}
Let $\mathcal{A}=T_s(A)$ be a $2$- and $(n-1)$-torsion free upper triangular matrix algebra, where $A$ is a unital algebra with center $\mathcal{Z}(A)$ and $s\geq 3$. Then every Lie $n$-derivation on $T_s(A)$ is standard.
\end{corollary}

\begin{corollary}\label{A_7}
Let $\mathcal{A}=T_2(A)$ be a $2$- and $(n-1)$-torsion free upper triangular matrix algebra, where $A$ is a unital algebra with center $\mathcal{Z}(A)$. Suppose that $A$ is a commutative algebra or a noncommutative prime algebra. Then every Lie $n$-derivation on $T_2(A)$ is standard.
\end{corollary}

\begin{remark}
We would mention that Corollary \ref{A_6} is not true in case that $s=2$. In \cite[Section 7]{Benkovic-Eremita}, the authors construct an example. Let $A$ be a $\mathbb{Z}_2-\mathrm{graded}$ algebra over $\mathcal{R}$, i.e., an algebra of the form $A=A_0+A_1$, where $A_0,A_1\subseteq A$ and multiplication in $A$ is such that $A_0A_0\subseteq A_0$, $A_1A_1\subseteq A_1$, $A_0A_1\subseteq A_1$ and $A_1A_0\subseteq A_1$. Suppose that $\mathcal{Z}(A)=A_0$ and $\varphi$ is a map on $\mathcal{A}=T_2(A)$. Define that $$\varphi\left(\begin{array}{cc}a_0+a_1&m_0+m_1\\&b_0+b_1\end{array}\right)=\left(\begin{array}{cc}b_1&-m_1\\&a_1\end{array}\right)$$
for each $a_0+a_1$, $m_0+m_1$ and $b_0+b_1$ in $A$. Then $\varphi$ is a Lie $n$-derivation and is not standard.
\end{remark}

Let $\mathcal{A}=\mathrm{alg}\mathcal{N}$ be a nest algebra, where $\mathcal{N}$ is a non-trivial nest in a Hilbert space $\mathcal{H}$ and $\mathrm{dim}\mathcal{H}\geq 2$. By \cite{Davidson}, $\mathrm{alg}\mathcal{N}$ can be viewed as a triangular algebra. Let $\mathcal{A}_0=E(\mathrm{alg}\mathcal{N})E$, where $E$ is the orthogonal projection on $\mathcal{N}$. Assume that $d$ is a central derivation of $\mathcal{A}_0$. Since $\mathcal{A}_0$ is also a nest, we can find an orthonormal projection $e$ onto $\mathcal{A}_0$, and view $\mathcal{A}_0$ as a triangular algebra. Since $d(e)\in\mathcal{Z}(\mathcal{A}_0)$ and $d(e)=d(ee)=ed(e)+d(e)e$, we have $d(e)=ed(e)=d(e)e=ed(e)e=0$. Similarly, we have $d(E-e)=0$. For each $m\in e\mathcal{A}_0(E-e)$, since $d(m)\in \mathcal{Z}(\mathcal{A}_0)$ and $d(m)=d(em(E-e))=ed(m)(E-e)$, we have $d(m)=0$. For each $a\in e\mathcal{A}_0e$ and $m\in e\mathcal{A}_0(E-e)$, since $0=d(am)=d(a)m$, we have $d(a)=0$. Similarly, we have $d(b)=0$ for each $b\in (E-e)\mathcal{A}_0(E-e)$. Then $d=0$, i.e., there exists no nonzero central derivation on $\mathcal{A}_0$. Since $\mathcal{Z}(\mathrm{alg}\mathcal{N})=\mathbb{C}I$, it obviously follows that (v) in Corollary \ref{A_5} holds. Thus, we have the following corollary which is also proved in \cite[Corollary 6.4]{Benkovic-Eremita}.

\begin{corollary}\label{A_8}
Let $\mathcal{A}=\mathrm{alg}\mathcal{N}$ be a nest algebra, where $\mathcal{N}$ is a non-trivial nest in a Hilbert space $\mathcal{H}$ and $\mathrm{dim}\mathcal{H}\geq 2$. Then every Lie $n$-derivation on $\mathrm{alg}\mathcal{N}$ is standard.
\end{corollary}

\subsection{\bf Unital prime algebras with nontrivial idempotents}
Let $\mathcal{A}$ be a unital prime algebra with a nontrivial idempotent $e$. If $exe\cdot e\mathcal{A}f=\{0\}$, i.e., $(exe)\mathcal{A}f=\{0\}$, then $exe=0$. And if $e\mathcal{A}f\cdot fxf=\{0\}$, i.e., $e\mathcal{A}(fxf)=\{0\}$, then $fxf=0$. Thus, $\mathcal{A}$ satisfies (\ref{1.6}). According to \cite[Corollary 4.5]{Benkovic-Sirovnik}, we obtain that every Jordan derivation of $\mathcal{A}$ is a derivation. If $e\mathcal{A}e$ is commutative, then $e\mathcal{A}e$ obviously satisfies (\ref{S.34}). Or if $e\mathcal{A}e$ is noncommutative, then $e\mathcal{A}e$ satisfies (\ref{S.34}) by \cite[Theorem 2]{Posner}, and it's not difficult to prove that $e\mathcal{A}e$ has no nonzero central ideal. According to Corollary \ref{S_13}, we have the following corollary.

\begin{corollary}\label{A_9}
Let $\mathcal{A}$ be a unital $2$- and $(n-1)$-torsion free prime algebra with a nontrivial idempotent $e$. If one of the following statements holds:
\begin{description}
\item[(i-1)] $e\mathcal{A}e=\mathcal{S}(e\mathcal{A}e)$ and $f\mathcal{A}f=\mathcal{S}(f\mathcal{A}f)$,
\item[(i-2)] $e\mathcal{A}e=\mathcal{S}(e\mathcal{A}e)$ and $\mathcal{Z}(e\mathcal{A}e)=e\mathcal{Z}(\mathcal{A})e$,
\item[(i-3)] $f\mathcal{A}f=\mathcal{S}(f\mathcal{A}f)$ and $\mathcal{Z}(f\mathcal{A}f)=f\mathcal{Z}(\mathcal{A})f$,
\item[(i-4)] $\mathcal{Z}(e\mathcal{A}e)=e\mathcal{Z}(\mathcal{A})e$ and $\mathcal{Z}(f\mathcal{A}f)=f\mathcal{Z}(\mathcal{A})f$,
\end{description}
and if one of the following statements also holds:
\begin{description}
\item[(ii-1)] $e\mathcal{A}e$ or $f\mathcal{A}f$ is noncommutative,
\item[(ii-2)] $e\mathcal{A}e$ or $f\mathcal{A}f$ has no nonzero central ideal,
\item[(ii-3)] $\mathcal{Z}(\mathcal{A})=\{\begin{array}{l|l}
                            a+b & a\in e\mathcal{Z}(\mathcal{A})e,~b\in f\mathcal{Z}(\mathcal{A})f,~
                                    am_0=m_0b
                            \end{array}\}$ for some $m_0\in e\mathcal{A}f$,
\item[(ii-4)] $\mathcal{Z}(\mathcal{A})=\{\begin{array}{l|l}
                            a+b & a\in e\mathcal{Z}(\mathcal{A})e,~b\in f\mathcal{Z}(\mathcal{A})f,~
                                    t_0a=bt_0
                            \end{array}\}$ for some $t_0\in f\mathcal{A}e$,
\item[(ii-5)] $\mathcal{A}$ satisfies (\ref{S.22-2}),
\end{description}
then every Lie $n$-derivation on $\mathcal{A}$ is standard.
\end{corollary}

Let $\mathcal{A}=B(X)$ be an algebra of all bounded linear operators, where $X$ is a Banach space over the complex field $\mathbb{C}$ and $\mathrm{dim}X\geq 2$. It's obvious that $B(X)$ is a unital prime algebra with a nontrivial idempotent $e$. Since the center $\mathcal{Z}(B(X))=\mathbb{C}I$, we have that $\mathcal{Z}(eB(X)e)=e\mathcal{Z}(B(X))e$ and $\mathcal{Z}(fB(X)f)=f\mathcal{Z}(B(X))f$. If $eB(X)e$ is commutative and $eB(X)f=\{0\}$, then $B(X)$ satisfies (\ref{S.22-2}). Or if $eB(X)e$ is commutative and $eB(X)f\neq \{0\}$, then we can choose an arbitrary nonzero element $m_0$ in $eB(X)f$. For arbitrary elements $\lambda\cdot eIe$ in $\mathcal{Z}(eB(X)e)$ and $\mu\cdot fIf$ in $\mathcal{Z}(fB(X)f)$ satisfying $(\lambda\cdot eIe)m_0=m_0(\mu\cdot fIf)$, since $(\lambda\cdot eIe)m_0=\lambda m_0$ and $m_0(\mu\cdot fIf)=\mu m_0$, we have that $\lambda=\mu$ and $\lambda\cdot eIe+\mu\cdot fIf=\lambda I\in\mathcal{Z}(B(X))$, which follows that (ii-3) in Corollary \ref{A_9} holds. According to Corollary \ref{A_9}, we have the following corollary which improves \cite[Theorem 1.1]{Lu-Liu}.

\begin{corollary}\label{A_10}
Let $\mathcal{A}=B(X)$ be an algebra of all bounded linear operators, where $X$ is a Banach space over the complex field $\mathbb{C}$ and $\mathrm{dim}X\geq 2$. Then every Lie $n$-derivation on $B(X)$ is standard.
\end{corollary}

Let $\mathcal{A}$ be a factor von Neumann algebra acting on a Hilbert space $\mathcal{H}$ with $\mathrm{dim}(\mathcal{A})\geq 2$. It's obvious that $\mathcal{A}$ is a unital prime algebra with nontrivial idempotents $p_1$, $p_2$. Let $\mathcal{A}_{ij}=p_i\mathcal{A}p_j$ where $1\leq i,j\leq 2$. Since the center $\mathcal{Z}(\mathcal{A})=\mathbb{C}I$, we have that $\mathcal{Z}(\mathcal{A}_{11})=p_1\mathcal{Z}(\mathcal{A})p_1$ and $\mathcal{Z}(\mathcal{A}_{22})=p_2\mathcal{Z}(\mathcal{A})p_2$. If $\mathcal{A}_{11}$ is commutative and $\mathcal{A}_{12}=\{0\}$, then $\mathcal{A}$ satisfies (\ref{S.22-2}). Or if $\mathcal{A}_{11}$ is commutative and $\mathcal{A}_{12}\neq \{0\}$, then we can choose an arbitrary nonzero element $m_0$ in $\mathcal{A}_{12}$. For arbitrary elements $\lambda\cdot p_1Ip_1$ in $\mathcal{Z}(\mathcal{A}_{11})$ and $\mu\cdot p_2Ip_2$ in $\mathcal{Z}(\mathcal{A}_{22})$ satisfying $(\lambda\cdot p_1Ip_1)m_0=m_0(\mu\cdot p_2Ip_2)$, since $(\lambda\cdot p_1Ip_1)m_0=\lambda m_0$ and $m_0(\mu\cdot p_2Ip_2)=\mu m_0$, we have that $\lambda=\mu$ and $\lambda\cdot p_1Ip_1+\mu\cdot p_2Ip_2=\lambda I\in\mathcal{Z}(\mathcal{A})$, which follows that (ii-3) in Corollary \ref{A_9} holds. According to Corollary \ref{A_9}, we have the following corollary.

\begin{corollary}\label{A_11}
Let $\mathcal{A}$ be a factor von Neumann algebra acting on a Hilbert space $\mathcal{H}$ with $\mathrm{dim}(\mathcal{A})\geq 2$. Then every Lie $n$-derivation on $\mathcal{A}$ is standard.
\end{corollary}

\subsection{\bf Von Neumann algebras without central summand of type $I_1$}
Let $\mathcal{A}$ be a von Neumann algebra with no central summand of type $I_1$. By \cite[Lemmas 4]{Miers-Lie_homo} and \cite[Lemma 1]{Miers-Lie_tri_der}, we know that $\mathcal{A}$ is a unital algebra with a nontrivial idempotent $p$ satisfying (\ref{1.6}) and (iii) in Corollary \ref{S_15}. Denote that $q=I-p$. Let $\mathcal{A}_{11}=p\mathcal{A}p$, $\mathcal{A}_{12}=p\mathcal{A}q$, $\mathcal{A}_{21}=q\mathcal{A}p$ and $\mathcal{A}_{22}=q\mathcal{A}q$. By \cite[Lemma 5]{Miers-Lie_homo}, we have that $\mathcal{Z}(\mathcal{A}_{11})=p\mathcal{Z}(\mathcal{A})p$ and $\mathcal{Z}(\mathcal{A}_{22})=q\mathcal{Z}(\mathcal{A})q$. Let $d_0$ be an arbitrary central inner derivation of $\mathcal{A}_{11}$, i.e., there exists an element $a'_{11}$ in $\mathcal{A}_{11}$ such that $d_0(a_{11})=[a'_{11},a_{11}]\in\mathcal{Z}(\mathcal{A}_{11})$ for each $a_{11}$ in $\mathcal{A}_{11}$. By the Kleinecke-Shirokov theorem \cite[Lemma 2.2]{Fosner-Wei-Xiao}, we have $d_0(a_{11})^2=0$ for each $a_{11}$ in $\mathcal{A}_{11}$. It follows form \cite[Lemma 1]{Miers-Lie_tri_der} that $d_0=0$. Thus, (i-4) in Corollary \ref{S_15} holds. Let $\mathcal{I}$ be the central ideal of $\mathcal{A}_{11}$. For each $a_{11}$ in $\mathcal{I}$, since $a_{11}\mathcal{A}_{11}\subseteq\mathcal{I}\subseteq\mathcal{Z}(\mathcal{A}_{11})$ and \cite[Lemma 5]{Bresar-Miers}, we have $a_{11}=0$. Thus, (ii-1) in Corollary \ref{S_15} holds. According to Corollary \ref{S_15}, we have the following corollary which is also proved in \cite[Theorem 2.3]{Fosner-Wei-Xiao}.

\begin{corollary}\label{A_12}
Let $\mathcal{A}$ be a von Neumann algebra with no central summand of type $I_1$. Then every Lie $n$-derivation on $\mathcal{A}$ is standard.
\end{corollary}



\end{document}